\documentclass[onefignum,onetabnum]{siamonline250211}

\usepackage{amssymb}
\usepackage{mathtools}
\usepackage{bm}
\usepackage{graphicx}
\usepackage{xcolor}
\usepackage{siunitx}
\usepackage{booktabs}
\usepackage{enumitem}

\newcommand{\N}[2]{\mathcal{N}(#1, #2)}
\newcommand{\prior}{\pi_{\mathrm{prior}}}
\newcommand{\priorm}{\pi_{\mathrm{prior}}^{m}}
\newcommand{\priors}{\pi_{\mathrm{prior}}^{s}}

\newcommand{\post}{\pi_{\mathrm{post}}^{\yobs}}
\newcommand{\E}[1]{\mathbb{E}\!\left[ #1 \right]}
\newcommand{\Ex}[2]{\mathbb{E}_{#1}\!\left[ #2 \right]}

\newcommand{\norm}[1]{\| #1 \|}
\newcommand{\Rnorm}[2]{\| #1 \|_{#2}}

\newcommand{\Signoise}{\bm{\Sigma}_{\mathrm{noise}}}
\newcommand{\yobs}{\bm{y}}
\newcommand{\Fwhite}{F}
\newcommand{\forig}{f}
\newcommand{\G}{\mathbf{G}}

\newcommand{\Fopt}{\tilde{\Fwhite}^*}
\newcommand{\FoptY}{\Fwhite_{\Py}^*}
\newcommand{\Gopt}{\G^*}
\newcommand{\sopt}{\bvecs^*}

\newcommand{\Bproj}{\mathcal{B}_{\mathrm{ridge}}}
\newcommand{\Aaff}{\mathcal{A}_{\mathrm{aff}}}

\newcommand{\grad}[1]{\nabla_{#1}}
\newcommand{\Hss}{\mathbf{H}_{ss}}

\newcommand{\btheta}{\bm{\theta}}
\newcommand{\bvecm}{\bm{m}}
\newcommand{\bvecs}{\bm{s}}
\newcommand{\bmu}{\bm{\mu}}
\newcommand{\balpha}{\bm{\alpha}}
\newcommand{\dobs}{\bm{d}_{\mathrm{obs}}}
\newcommand{\bars}{\bar{\bvecs}}
\newcommand{\barm}{\bar{\bvecm}}

\newcommand{\Phiy}{\Phi_{y}}
\newcommand{\Phim}{\Phi_{m}}
\newcommand{\Phis}{\Phi_{s}}

\newcommand{\Cy}{\bm{C}_{\bm{y}}}
\newcommand{\Qproj}{\hat{\mathbf{Q}}}
\newcommand{\Qsurr}{\tilde{\mathbf{Q}}}
\newcommand{\qproj}{\hat{\bm{q}}}
\newcommand{\qsurr}{\tilde{\bm{q}}}

\newcommand{\Gm}{\G(\bvecm)}

\DeclareMathOperator{\tr}{tr}

\DeclareMathOperator{\diag}{diag}

\newcommand{\Rs}{\bm{R}_s}
\newcommand{\Rm}{\bm{R}_m}
\newcommand{\Py}{\mathcal{P}_y}
\newcommand{\Ptheta}{\mathcal{P}_\theta}
\newcommand{\Pm}{\mathcal{P}_m}
\newcommand{\Ps}{\mathcal{P}_s}

\newcommand{\Robs}{\mathcal{R}_{\mathrm{obs}}}
\newcommand{\Rsrcaff}{\mathcal{R}_{\mathrm{src,aff}}}
\newcommand{\Cmedaff}{\mathcal{C}_{\mathrm{med,aff}}}
\newcommand{\Tmedaff}{\mathcal{T}_{\mathrm{med,aff}}}
\newcommand{\Hmmaff}{\mathbf{H}_{mm}^{\mathrm{aff}}}
\newcommand{\Rsurraff}{\mathcal{R}_{\mathrm{surr,aff}}}

\newsiamremark{remark}{Remark}
\usepackage{aliascnt}
\newaliascnt{assumption}{theorem}
\newtheorem{assumption}[assumption]{Assumption}
\aliascntresetthe{assumption}

\crefname{assumption}{assumption}{assumptions}
\Crefname{assumption}{Assumption}{Assumptions}

\headers{Partially Linear Bayesian Inverse Problems}{Z. Zhang, D. Luo, T. O'Leary-Roseberry, U. Villa, and O. Ghattas}

\title{Structure-Preserving Dimension Reduction and Surrogates for \\
Partially Linear Bayesian Inverse Problems%
\thanks{
This material is based upon work supported by the
U.S. National Science Foundation under awards 2425922 and 2313033.
}}

\author{ Ziheng Zhang\thanks{Oden Institute for Computational Engineering and Sciences, The University of Texas at Austin, Austin, TX (\email{ziheng@utexas.edu}).}
\and Dingcheng Luo\thanks{Faculty of Science, School of Mathematical Sciences, Queensland University of Technology, Brisbane, Queensland, Australia (\email{dingcheng.luo@qut.edu.au}).}
\and Thomas O'Leary-Roseberry\thanks{Department of Mathematics, The Ohio State University, Columbus, OH (\email{tomor@math.osu.edu}).} 
\and Umberto Villa\thanks{Department of Biomedical Engineering and Oden Institute for Computational Engineering and Sciences, The University of Texas at Austin, Austin, TX (\email{uvilla@oden.utexas.edu}).}
\and Omar Ghattas\thanks{Oden Institute for Computational Engineering and Sciences and Walker Department of Mechanical Engineering, The University of Texas at Austin, Austin, TX (\email{omar@oden.utexas.edu}).} }

\begin{document}

\maketitle

\begin{abstract}
Many Bayesian inverse problems are characterized by two parameter vectors with distinct roles. A medium property, operator coefficient, or geometry block enters the forward model nonlinearly, whereas a volumetric, initial condition, or boundary condition source enters linearly. General-purpose nonlinear surrogates must relearn this known source dependence from data and may not preserve the correct source scaling outside the training regime. We introduce a structure-preserving partially linear (SPPL) framework that constructs separate reduced spaces for outputs, (linearly appearing) sources, and (nonlinearly appearing) medium variables and represents the reduced forward map as an affine function of the retained source coordinates, with coefficients that depend on the medium variables. Here, “affine” refers to the retained, mean-centered source coordinates: the physical forward map remains linear in the source, while the affine offset is induced by the source-prior mean and vanishes when that mean is zero. The analysis first identifies the optimal blockwise ridge approximation and then derives the optimal affine-in-source approximation under independent medium and source priors. A direct decomposition of the forward error into output, source, medium, and surrogate contributions yields blockwise spectral criteria and prior-predictive expected-KL control. For a Gaussian source prior, the optimal affine-in-source approximation coincides with the optimal blockwise ridge approximation. In a coseismic numerical example, a three-model controlled ablation on common block-reduced coordinates separates the effects of coefficient-matrix supervision and structure-preserving affine-in-source parameterization. Reduced-space surrogate MCMC provides orders-of-magnitude online acceleration relative to high-fidelity posterior sampling.
\end{abstract}



\section{Introduction}
\label{sec:introduction}

Many inverse problems require simultaneous recovery of two parameter block vectors (hereafter referred to as {\em blocks}) with different roles: one block enters the forward model nonlinearly, whereas the other enters linearly. Such mixed linear--nonlinear structure arises, for example, in inverse heat problems and simultaneous structural-parameter/load identification \cite{zhao2020heat_coefficient_initial_state,prells1999variable,tang2022dynamic_load_unknown_parameters}, as well as in geophysical inverse problems that jointly estimate medium properties and sources. We call the parameter-to-observable (PtO) map $F$ partially linear when \( \Fwhite(\bvecm,\bvecs)=\G(\bvecm)\bvecs\), where the medium variable \(\bvecm\) enters nonlinearly through \(\G(\bvecm)\), while the source variable \(\bvecs\) enters linearly. Coseismic fault slip inversion with uncertain material properties provides a motivating example in heterogeneous linear elasticity: the material field changes the elasticity operator, whereas fault slip enters linearly through the fault interface. Puel et al.\ developed a mixed finite-element framework for fault slip \cite{Puel22} and subsequently applied an adjoint formulation to joint inversion of heterogeneous material properties and fault slip \cite{puel2024adjoint_joint}.

A Bayesian formulation quantifies uncertainty in both parameter blocks and their posterior dependence \cite{stuart2010inverse,dashti2017bayesian}. However, posterior exploration requires repeated evaluations of an expensive forward model, motivating function-space and likelihood-informed MCMC \cite{cotter2013mcmc,cui2016dimension}, geometric MCMC for PDE-constrained inverse problems \cite{buithanh2014solving}, measure transport methods \cite{baptista2025approximation}
scalable low-rank computational frameworks \cite{buithanh2013framework}, and surrogate or reduced-order approximations \cite{frangos2010surrogate_rom_inverse,li2014adaptive_surrogates_bip,gao2024adaptive_operator_learning,
cao2026lazydino,cao2025derivative}. A generic projected nonlinear surrogate does not encode the known linear source action and therefore must infer it from training data; in particular, the scaling relation \(F(m,cs)=cF(m,s)\) of the original partially linear map is not enforced by the surrogate architecture. More generally, extrapolation outside the training distribution remains challenging for neural operators \cite{zhu2023reliable_extrapolation}.

To exploit this structure, we propose the structure-preserving partially linear (SPPL) framework, which combines blockwise dimension reduction with a structure-preserving surrogate representation. SPPL constructs separate reduced spaces for the output, source, and medium variables, and represents the reduced forward map as affine in the retained, mean-centered source coordinates while learning only its medium-dependent coefficients. In these coordinates, the original linear source action becomes affine because of the source-prior mean, with the zero-mean case reducing to a linear map. The resulting surrogate can then be used within standard posterior-sampling strategies; in this work, we use it for joint reduced-space MCMC.

For a non-Gaussian source prior, 
the optimal reduced map does not, in general, retain the affine dependence on the source.
We therefore first derive the unrestricted retained-coordinate optimum as a benchmark and then the best approximation within the affine-in-source class used by SPPL.

\paragraph{Main contributions}
These are as follows.
\begin{enumerate}[label=\textnormal{(\roman*)}, leftmargin=1.1cm]
    \item We characterize the optimal blockwise ridge approximation and its medium/source conditional-mean factorization under independent priors.

    \item We derive the best affine-in-source approximation for source priors with finite, positive-definite covariance, together with a direct output--source--medium--surrogate error decomposition, spectral reduction criteria, and prior-predictive expected-KL control.

    \item We derive a medium-only coefficient loss and a PDE factorization-reuse workflow for constructing the medium diagnostic and coefficient targets. For Gaussian sources, the optimal affine and blockwise ridge approximations coincide.

    \item We validate SPPL on quadratic benchmarks and a PDE-based coseismic inversion, using a controlled three-model ablation to separate source parameterization from training supervision and evaluating posterior fidelity, source-scale robustness, and reduced-space MCMC performance.
\end{enumerate}

\paragraph{Related work}
Our work connects most directly to mixed linear--nonlinear optimization, derivative-informed dimension reduction, and structure-aware surrogate modeling.

First, separable nonlinear least squares and variable projection exploit mixed linear--nonlinear structure by eliminating or conditionally optimizing the linear variables while updating the nonlinear variables \cite{golub1973differentiation,golub2003separable,chung2010efficient,espanol2023variable}; related source-estimation strategies have also been used in seismic inversion \cite{rickett2013variable_projection_unknown_source,fang2018source_estimation_wri}. SPPL exploits the same structural separation at the reduced forward-model level, preserving the linear source action explicitly while learning the nonlinear medium dependence and treating both parameter blocks as uncertain.

Second, SPPL is related to likelihood-informed, active-subspace, low-rank posterior approximation, and gradient-based dimension reduction, which use derivative, Fisher, Hessian, or Poincar\'e information to identify informative parameter or observable directions \cite{cui2014likelihood,cui2016dimension,cui2016scalable,zahm2020gradient,cui2021data,zahm2022certified,constantine2016accelerating,spantini2015optimal,baptista2022gradient,chen2025coupled}. Unlike approaches that reduce a single joint parameter, SPPL derives separate output, source, and medium criteria directly from the partially linear map structure.

Third, SPPL builds on surrogate and reduced-order modeling for Bayesian inverse problems and parametric PDEs, including reduced-order, adaptive, Gaussian-process, neural, and operator-learning approaches \cite{frangos2010surrogate_rom_inverse,li2014adaptive_surrogates_bip,yan2020adaptive_dnn_bip,gao2024adaptive_operator_learning,bai2024gp_bayesian_inverse,lu2021learning,kovachki2023neural,zhu2023reliable_extrapolation}. Related work establishes approximation guarantees for Bayesian inverse problems \cite{cotter2010approximation_bip_pdes,stuart2018posterior_consistency_gp}, while other approaches use approximate models within MCMC \cite{christen2005markov,conrad2018parallel_local_approx_mcmc}. Derivative-informed projected neural networks (DIPNets), used below as a baseline, combine derivative-informed reduced coordinates with nonlinear projected surrogates \cite{oleary2022derivative}. Related derivative-informed operator-learning developments address high-dimensional parametric derivative learning, approximation-error analysis, and acceleration of geometric MCMC \cite{oleary2024derivative,luo2026dimension,cao2025derivative}; these approaches do not impose the affine retained-source action used by SPPL. Most closely, NEMO \cite{pham2026nemo} learns a reduced medium-dependent matrix operator while preserving linear action on the inversion parameters. Its formulation treats the nonlinear model parameters as externally characterized during inversion, whereas SPPL treats both parameter blocks as uncertain and derives the optimal blockwise ridge and affine factorizations, blockwise reduction criteria, and Bayesian error control.

Finally, in the terminology of theory-guided and physics-guided scientific learning \cite{karpatne2017theory_guided,faroughi2024physics_guided_encoded}, SPPL is physics-encoded: affine dependence on the retained source coordinates is imposed by the surrogate class rather than encouraged through a soft penalty.

\section{Partially Linear Bayesian Inverse Problems}
\label{sec:partial_linear_formulation}

We first recall the generic Bayesian inverse problem in whitened form, before specializing to the partially linear structure considered in this work.

\subsection{Whitened Bayesian formulation and forward-error control}
\label{sec:whitened_formulation}

Let \(\btheta\) denote the unknown parameter and suppose
\(\dobs=\forig(\btheta)+\bm{\eta}\), where
\(\bm{\eta}\sim\N{\bm{0}}{\Signoise}\) and \(\Signoise\) is symmetric positive definite. Define
\begin{equation*}
    \yobs \coloneqq \Signoise^{-1/2}\dobs,
    \qquad
    \Fwhite(\btheta) \coloneqq \Signoise^{-1/2}\forig(\btheta),
    \qquad
    \bm{\epsilon} \coloneqq \Signoise^{-1/2}\bm{\eta}.
\end{equation*}
Then \(\yobs=\Fwhite(\btheta)+\bm{\epsilon}\), with
\(\bm{\epsilon}\sim\N{\bm{0}}{I}\), and
\[
    \post(\btheta)
    \propto
    \exp\!\left(-\frac12\norm{\Fwhite(\btheta)-\yobs}^2\right)
    \prior(\btheta).
\]
Here \(\norm{\cdot}\) is the Euclidean norm, and
\(\Rnorm{v}{R}^2\coloneqq v^T Rv\) for any SPD matrix \(R\).

A surrogate \(\tilde{\Fwhite}\) defines the analogous posterior
\[
    \tilde{\pi}_{\mathrm{post}}^{\yobs}(\btheta)
    \propto
    \exp\!\left(-\frac12\norm{\tilde{\Fwhite}(\btheta)-\yobs}^2\right)
    \prior(\btheta).
\]
Under the prior-predictive law obtained by drawing
\(\btheta\sim\prior\) and \(\bm{\epsilon}\sim\mathcal N(0,I)\) independently and setting
\(\yobs=\Fwhite(\btheta)+\bm{\epsilon}\), the data-free bound of
Cui and Zahm~\cite{cui2021data} gives
\begin{equation}\label{eq:prior_predictive_kl_forward_bound}
    \Ex{\yobs}{
        D_{\mathrm{KL}}\!\left(
            \post \,\|\, \tilde{\pi}_{\mathrm{post}}^{\yobs}
        \right)
    }
    \le
    \frac12
    \Ex{\btheta\sim\prior}{
        \norm{\Fwhite(\btheta)-\tilde{\Fwhite}(\btheta)}^2 } = \frac{1}{2}\|\Fwhite - \tilde{\Fwhite}\|^2_{L^2_{\prior}}.
\end{equation}
Thus prior-weighted forward accuracy controls the prior-predictive posterior error.

\subsection{Partially linear model and priors}
\label{sec:partially_linear_model}

Let \(\btheta=(\bvecm,\bvecs)\), with
\(\bvecm\in\mathbb R^{d_m}\), \(\bvecs\in\mathbb R^{d_s}\), and
\(\Fwhite(\bvecm,\bvecs)\in\mathbb R^{n_{\mathrm{obs}}}\). In a canonical PDE setting,
\[
    \mathcal L(\bvecm)\bm x=\mathcal S\bvecs,
    \qquad
    \dobs=\mathcal B\bm x+\bm\eta,
\]
where \(\bm x\) denotes the state and \(\mathcal S\) is a fixed linear
source-injection operator. The state equation is linear in \(\bm x\) and
\(\bvecs\), while the operator depends nonlinearly on \(\bvecm\). Defining
\[
\Gm
\coloneqq
\Signoise^{-1/2}\mathcal B\mathcal L(\bvecm)^{-1}\mathcal S
\in\mathbb R^{n_{\mathrm{obs}}\times d_s},
\qquad
\Fwhite(\bvecm,\bvecs)=\Gm\bvecs,
\]
gives the partially linear whitened map. The term ``partially linear'' refers
to this PtO structure, not to the differential operator itself.

We assume independent priors,
\(\prior(\bvecm,\bvecs)=\priorm(\bvecm)\priors(\bvecs)\), so the posterior is given by
\[
 \post(\bvecm,\bvecs)
 \propto
 \exp\!\left(-\frac12\norm{\Gm\bvecs-\yobs}^2\right)
 \priorm(\bvecm)\priors(\bvecs).
\]
For Gaussian priors
\(\priorm=\N{\barm}{\bm C_m}\) and
\(\priors=\N{\bars}{\bm C_s}\), with SPD covariance matrices
\(\bm C_m\) and \(\bm C_s\), the negative log-posterior is
\begin{equation*}
\mathcal{J}^{\yobs}(\bvecm,\bvecs)
=
\frac12\norm{\Gm\bvecs-\yobs}^2
+\frac12\Rnorm{\bvecm-\barm}{\bm C_m^{-1}}^2
+\frac12\Rnorm{\bvecs-\bars}{\bm C_s^{-1}}^2.
\end{equation*}

With this partially linear Bayesian formulation in place, we next exploit its block structure to develop separate output, source, and medium reductions and derive structure-preserving approximation and error control results.

\section{Structure-Preserving Reduction Theory}
\label{sec:sppl_reduction_theory}

We now exploit the partially linear structure
\(\Fwhite(\bvecm,\bvecs)=\G(\bvecm)\bvecs\)
to derive separate reductions for the output, source, and medium variables. Conditional expectation characterizes the \(L^2_{\prior}\)-optimal retained-coordinate approximations, while gradient-based Poincar\'e bounds provide a computable criterion for medium reduction~\cite{zahm2020gradient,cui2021data}.

\subsection{Basic setting and reduced coordinates}
\label{sec:reduced_coordinates}

\begin{assumption}[Standing assumptions]
\label{ass:standing_assumptions}
The following conditions hold.
\begin{enumerate}[label=\textnormal{(\roman*)}, leftmargin=1.1cm, itemsep=2pt, topsep=2pt]
\item The medium and source variables are independent under the prior,
\(\prior(\bvecm,\bvecs)=\priorm(\bvecm)\priors(\bvecs)\).

\item The map \(\bvecm\mapsto\G(\bvecm)\) is continuously differentiable and
\begin{equation*}
\Ex{\bvecm\sim\priorm}{
\|\Gm\|_{\rm F}^2+
\sum_{\ell=1}^{d_m}\|\partial_{\bvecm_\ell}\Gm\|_{\rm F}^2
}<\infty.
\end{equation*}
Whenever an expected-KL bound is invoked, the exact and approximate posterior normalizing constants are finite for almost every \(\yobs\) under the prior-predictive distribution.

\item The source prior has finite second moment and positive-definite covariance. Define
\begin{equation*}
\bars\coloneqq\Ex{\bvecs\sim\priors}{\bvecs},
\qquad
\bm C_s\coloneqq
\Ex{\bvecs\sim\priors}{
(\bvecs-\bars)(\bvecs-\bars)^T},
\qquad
\Rs\coloneqq\bm C_s^{-1}.
\end{equation*}
\end{enumerate}
\end{assumption}

These assumptions imply the square-integrability of \(\G(\bvecm)\bvecs\). The differentiability and integrability conditions will also ensure that the medium Jacobian quantities introduced later are square integrable.

With \(\bm C_s\) now defined, the source and medium blocks use distinct metric matrices. The source metric is \(\Rs=\bm C_s^{-1}\). For the medium block, we fix a centering vector \(\barm\in\mathbb R^{d_m}\) and an SPD matrix \(\Rm\), which defines the weighted inner product used for the retained medium coordinates. For a Gaussian medium prior \(\priorm=\N{\barm}{\bm C_m}\), the canonical choice is the prior precision, \(\Rm=\bm C_m^{-1}\).

\paragraph{Reduced-coordinate conventions}
With these quantities fixed, let
\(\Phiy\in\mathbb R^{n_{\mathrm{obs}}\times r_y}\),
\(\Phim\in\mathbb R^{d_m\times r_m}\), and
\(\Phis\in\mathbb R^{d_s\times r_s}\) 
be reduced bases of the observable, medium, and source spaces that
satisfy
\[
\Phiy^T\Phiy=I_{r_y},
\qquad
\Phim^T\Rm\Phim=I_{r_m},
\qquad
\Phis^T\Rs\Phis=I_{r_s}.
\]
We define projection and encoding operators
\[
\Py=\Phiy\Phiy^T,
\qquad
\Phim^*=\Phim^T\Rm,
\quad \Pm=\Phim\Phim^*,
\qquad
\Phis^*=\Phis^T\Rs,
\quad \Ps=\Phis\Phis^*.
\]
Here, \(\Py\) is Euclidean orthogonal, while \(\Pm\) and \(\Ps\) are
\(\Rm\)- and \(\Rs\)-orthogonal. The retained coordinates are
\[
    \bmu=\Phim^*(\bvecm-\barm),
    \qquad
    \balpha=\Phis^*(\bvecs-\bars).
\]
By the source-moment definitions and the normalization
\(\Phis^T\Rs\Phis=I_{r_s}\), the retained source coordinates satisfy
\( \E{\balpha}=0, \E{\balpha\balpha^T}=I_{r_s}\). Because \(\barm\) and \(\bars\) are fixed,
\[
\Pm\bvecm=\Pm\barm+\Phim\bmu,
\qquad
\Ps\bvecs=\Ps\bars+\Phis\balpha.
\]
Hence \(\sigma(\Pm\bvecm)=\sigma(\bmu)\) and \(\sigma(\Ps\bvecs)=\sigma(\balpha)\): conditioning on \(\Pm\bvecm\) or \(\Ps\bvecs\) below is therefore equivalent to conditioning on the corresponding centered retained coordinates. Finally, we set
\(\Ptheta=\diag(\Pm,\Ps)\).

\subsection{Optimal blockwise ridge approximation}
\label{sec:blockwise_ridge_approximation}

For fixed projectors, let
\(\bvecm_r=\Pm\bvecm\), \(\bvecs_r=\Ps\bvecs\) denote the reduced coordinates. 
We also define the following conditional expectations:
\begin{equation*}
\Gopt(\bvecm_r)
\coloneqq\E{\Gm\mid\Pm\bvecm=\bvecm_r},
\qquad
\sopt(\bvecs_r)
\coloneqq\E{\bvecs\mid\Ps\bvecs=\bvecs_r}.
\end{equation*}
All conditional-expectation identities are understood almost surely.

\begin{definition}[Blockwise ridge approximation class]
\label{def:blockwise_ridge_approximation_class}
Let \(\Bproj(\Pm,\Ps,\Py)\) denote, up to almost-sure equality, the class of maps of the form
\( \tilde{\Fwhite}(\Ptheta\btheta) = \Py\,g(\Pm\bvecm,\Ps\bvecs) \), where \(g\) is square integrable. Thus each element within the class is a function that depends only on the retained medium and source variables and takes values in \(\operatorname{Im}(\Py)\).
\end{definition}

\begin{theorem}[Optimal blockwise ridge approximation]
\label{thm:optimal_blockwise_ridge_approximation}
The \(L^2_{\prior}\)-optimal element of \\
\(\Bproj(\Pm,\Ps,\Py)\) is
\begin{equation*}
\FoptY(\Ptheta\btheta)
\coloneqq
\E{\Py\Fwhite(\btheta)\mid\Ptheta\btheta}
=
\Py\,\Gopt(\bvecm_r)\,\sopt(\bvecs_r).
\end{equation*}
\end{theorem}

See \cref{app:proof_blockwise_ridge_approximation} for the proof.

\subsection{Practical affine-in-source class}
\label{sec:affine_in_source_class}

This optimal blockwise ridge approximation preserves the operator--source factorization, but for a non-Gaussian source prior the retained conditional mean \(\sopt(\Ps\bvecs)\) can be nonlinear. The practical SPPL construction therefore restricts the retained source dependence to an affine class.

\begin{definition}[Affine-in-source approximation class]
\label{def:affine_approximation_class}
Let \(\Aaff(\Pm,\Ps,\Py)\) be the subspace of
\(\Bproj(\Pm,\Ps,\Py)\) consisting of maps
\begin{equation*}
\tilde{\Fwhite}_{\mathrm{aff}}(\Ptheta\btheta)
=
a(\Pm\bvecm)+B(\Pm\bvecm)\balpha,
\end{equation*}
where
\( a:\operatorname{Im}(\Pm)\to\mathbb R^{n_{\mathrm{obs}}}, B:\operatorname{Im}(\Pm)\to
\mathbb R^{n_{\mathrm{obs}}\times r_s} \) are measurable and satisfy
\[
\Ex{\bvecm\sim\priorm}{
\|a(\Pm\bvecm)\|^2+\|B(\Pm\bvecm)\|_{\mathrm F}^2}<\infty,
\]
together with \(\Py a(z)=a(z)\), \(\Py B(z)=B(z)\), \(z\in\operatorname{Im}(\Pm)\).
Thus \(a\) is the medium-dependent affine offset and \(B\) is the
medium-dependent linear map acting on the retained source coordinates
\(\balpha\in\mathbb R^{r_s}\).
\end{definition}

\begin{theorem}[Optimal affine-in-source approximation]
\label{thm:optimal_affine_in_source}
The \(L^2_{\prior}\)-projection of \(\Fwhite\) onto
\(\Aaff(\Pm,\Ps,\Py)\) is
\begin{equation*}
\tilde{\Fwhite}_{\mathrm{aff}}^*(\Ptheta\btheta)
=
\Py\,\Gopt(\bvecm_r)\,
\widehat{\bvecs}_{\mathrm{aff}}(\balpha),
\end{equation*}
with
\( \widehat{\bvecs}_{\mathrm{aff}}(\balpha)
\coloneqq \bars+\Phis\balpha = \bars+\Ps(\bvecs-\bars) \).
Equivalently, this is the \(L^2_{\prior}\)-projection of the optimal blockwise ridge approximation
\(\FoptY\) onto \(\Aaff(\Pm,\Ps,\Py)\).
\end{theorem}

See \Cref{app:proof_optimal_affine_in_source} for the proof.

Thus the optimal affine-in-source approximation depends on the source prior only through its mean and covariance. For Gaussian source priors, restricting the retained source dependence to
the affine class introduces no additional approximation error.

\begin{corollary}[Exactness of the affine restriction for Gaussian source priors]
\label{cor:gaussian_affine_restriction_exactness}
If \(\priors=\N{\bars}{\Rs^{-1}}\), then
\begin{equation}
\label{eq:gaussian_source_conditional_mean}
\sopt(\Ps\bvecs)
=
\E{\bvecs\mid\Ps\bvecs}
=
\bars+\Ps(\bvecs-\bars)
=
\bars+\Phis\balpha.
\end{equation}
Consequently,
\begin{equation}
\label{eq:gaussian_affine_equals_blockwise_ridge}
\widehat{\bvecs}_{\mathrm{aff}}(\balpha)=\sopt(\Ps\bvecs),
\qquad
\tilde{\Fwhite}_{\mathrm{aff}}^*(\Ptheta\btheta)
=
\FoptY(\Ptheta\btheta)
=
\Py\Gopt(\bvecm_r)(\bars+\Phis\balpha).
\end{equation}
\end{corollary}

See \Cref{app:proof_gaussian_affine_restriction_exactness} for the proof.

We now work directly with this practical affine class. Its prior-averaged forward
error admits a decomposition whose terms determine the reduction criteria used
in the algorithm.

\subsection{Direct affine error decomposition and exact residuals}
\label{sec:affine_error_decomposition}

For the affine approximation class, we now decompose the prior-averaged forward error into contributions from output truncation, source reduction, medium reduction, and surrogate approximation. The first two contributions admit exact spectral trace representations. Define the uncentered forward-output second-moment matrix
\begin{equation*}
\Cy
=
\Ex{\bvecm\sim\priorm}{
\Gm\left(\Ex{\bvecs\sim\priors}{\bvecs\bvecs^T}\right)\Gm^T},
\end{equation*}
and, after fixing \(\Py\), the source diagnostic matrix
\begin{equation*}
\Hss(\Py)
\coloneqq
\Ex{\bvecm\sim\priorm}{\Gm^T\Py\Gm}.
\end{equation*}
For the projected whitened Gaussian likelihood, \(\Gm^T\Py\Gm\) is the source Gauss--Newton Hessian (equivalently, Fisher information matrix) conditional on \(\bvecm\); hence \(\Hss(\Py)\) is its prior average.

\begin{proposition}[Direct affine error decomposition]
\label{prop:affine_error_decomposition}
For any
\(\tilde{\Fwhite}_{\mathrm{aff}}\in\Aaff(\Pm,\Ps,\Py)\),
\begin{equation}
\label{eq:affine_direct_mse_split}
\Ex{\btheta\sim\prior}{
\left\|\Fwhite(\btheta)-
\tilde{\Fwhite}_{\mathrm{aff}}(\Ptheta\btheta)\right\|^2}
=
\Robs+\Rsrcaff+\Cmedaff+\Rsurraff,
\end{equation}
where
\begin{align}
\Robs
&\coloneqq
\Ex{\prior}{\|(I-\Py)\Fwhite\|^2}
=
\tr\!\left((I-\Py)\Cy(I-\Py)\right),
\label{eq:Robs}\\
\Rsrcaff
&\coloneqq
\Ex{\prior}{
\left\|\Py\Gm
\left(\bvecs-\widehat{\bvecs}_{\mathrm{aff}}(\balpha)\right)
\right\|^2} =
\tr\!\left(
\Rs^{-1}(I-\Ps^T)\Hss(\Py)(I-\Ps)
\right),
\label{eq:Rsrc_aff_exact}\\
\Cmedaff
&\coloneqq
\Ex{\prior}{
\left\|
\Py\left(\Gm-\Gopt(\bvecm_r)\right)
\widehat{\bvecs}_{\mathrm{aff}}(\balpha)
\right\|^2},
\label{eq:Cmed_aff_equiv}\\
\Rsurraff
&\coloneqq
\Ex{\prior}{
\left\|
\tilde{\Fwhite}_{\mathrm{aff}}^*(\Ptheta\btheta)
-
\tilde{\Fwhite}_{\mathrm{aff}}(\Ptheta\btheta)
\right\|^2}.
\label{eq:Rsurr_aff}
\end{align}
\end{proposition}

See \Cref{app:proof_affine_error_decomposition} for the proof.

The decomposition reveals an important distinction between the three reduction blocks. The output residual \(\Robs\) and source residual \(\Rsrcaff\) are exact quantities already expressed in spectral trace form, so their rank-constrained minimizing subspaces can be obtained directly. The medium residual \(\Cmedaff\), by contrast, is a conditional-approximation error and requires an additional Poincar\'e bound before an analogous spectral criterion becomes available.

\subsection{Spectral reduction criteria and medium Poincar\'e control}
\label{sec:medium_reduction_criteria}

To convert the trace residuals above into rank-constrained subspace selection rules, we use the following standard generalized-eigenvalue trace-minimization identity. Let \(\mathbf H\) be symmetric positive semidefinite and \(\mathbf R\) symmetric positive definite, and let
\(\mathbf H\bm\phi_i=\lambda_i\mathbf R\bm\phi_i\), with eigenvalues ordered nonincreasingly and eigenvectors chosen \(\mathbf R\)-orthonormal. Then~\cite[Prop.~2.6]{zahm2020gradient}
\begin{equation}\label{eq:trace_trailing_ev_equality}
\min_{\substack{P^2=P\\ \operatorname{rank}(P)=r}}
\tr\!\left(
    \mathbf R^{-1}(I-P^T)\mathbf H(I-P)
\right)
=
\sum_{i>r}\lambda_i,
\end{equation}
and the minimum is attained by the \(\mathbf R\)-orthogonal projector onto the dominant generalized eigenspace. Accordingly, whenever a projector is selected by a generalized eigenproblem below, we use this weighted-orthogonal representative.

Applying \eqref{eq:trace_trailing_ev_equality} with \((\mathbf H,\mathbf R)=(\Cy,I)\) shows that, at fixed output rank, the leading \(r_y\) eigenvectors of \(\Cy\), i.e., the output proper orthogonal decomposition (POD) modes, minimize the discarded output energy \(\Robs\). After \(\Py\) is fixed, applying the identity with \((\mathbf H,\mathbf R)=(\Hss(\Py),\Rs)\) shows that the dominant generalized eigenspace of \((\Hss(\Py),\Rs)\) minimizes \(\Rsrcaff\).

It remains to obtain a corresponding computable criterion for the medium block. For this purpose, we require the medium prior and metric \(\Rm\) to satisfy the following subspace Poincar\'e inequality, which connects medium truncation to derivative information.

\begin{assumption}[Medium subspace Poincar\'e inequality]
\label{ass:medium_poincare}
The SPD Poincar\'e metric matrix \(\Rm\) fixed in \Cref{sec:reduced_coordinates} is such that the medium prior
\(\priorm\) satisfies the following subspace Poincar\'e inequality:
for every continuously differentiable vector-valued \(h\) and every projector \(\Pm\),
\begin{equation*}
\begin{aligned}
&\Ex{\bvecm\sim\priorm}{
\norm{h(\bvecm)-\E{h(\bvecm)\mid \Pm\bvecm}}^2}
\\
&\qquad\le
\tr\!\left(
\Rm^{-1}(I-\Pm^T)
\Ex{\bvecm\sim\priorm}{
(\grad{\bvecm}h)^T(\grad{\bvecm}h)}
(I-\Pm)
\right).
\end{aligned}
\end{equation*}
\end{assumption}

For any fixed source vector \(v\in\mathbb R^{d_s}\), define the medium Jacobian
\begin{equation*}
    J_v(\bvecm)
    \coloneqq
    \nabla_{\bvecm}\bigl(\G(\bvecm)v\bigr)
    \in \mathbb R^{n_{\mathrm{obs}}\times d_m}.
\end{equation*}
By \Cref{ass:standing_assumptions}, the Jacobian quantities used below are square integrable.

To control the remaining medium term \(\Cmedaff\), recall that \(\Phis\) is the \(\Rs\)-orthonormal basis for the retained source subspace, with corresponding projector \(\Ps=\Phis\Phis^T\Rs\). Writing \(\Phis=[\bm\phi_1,\ldots,\bm\phi_{r_s}]\), define the affine medium diagnostic matrix
\begin{align}
\Hmmaff(\Py,\Ps)
&\coloneqq
\Ex{(\bvecm,\bvecs)\sim\prior}{
J_{\widehat{\bvecs}_{\mathrm{aff}}(\balpha)}(\bvecm)^T
\Py
J_{\widehat{\bvecs}_{\mathrm{aff}}(\balpha)}(\bvecm)}
\label{eq:Hmmaff_definition}\\
&=
\Ex{\bvecm\sim\priorm}{
J_{\bars}(\bvecm)^T\Py J_{\bars}(\bvecm)
+
\sum_{i=1}^{r_s}
J_{\bm\phi_i}(\bvecm)^T\Py J_{\bm\phi_i}(\bvecm)}.
\label{eq:Hmmaff_moment_expansion}
\end{align}
Thus \(\Hmmaff(\Py,\Ps)\) is the prior average of the medium Gauss--Newton Hessian (equivalently, Fisher information matrix) for the affine retained-source model. The second identity uses only
\(\E{\balpha}=0\) and
\(\E{\balpha\balpha^T}=I_{r_s}\); hence \(\Hmmaff(\Py,\Ps)\) depends on the source prior only through its mean, covariance, and retained source space.

\begin{proposition}[Affine medium bound and expected-KL control]
\label{prop:affine_trace_kl_bounds}
The affine medium residual satisfies
\begin{equation*}
\Cmedaff \le \Tmedaff \coloneqq
\tr\!\left( \Rm^{-1}(I-\Pm^T)
\Hmmaff(\Py,\Ps) (I-\Pm) \right).
\end{equation*}
Consequently,
\begin{equation}
\label{eq:affine_expected_kl_bound}
\Ex{\yobs}{
D_{\mathrm{KL}}\!\left(
\post\,\|\,\tilde{\pi}_{\mathrm{post}}^{\yobs}
\right)} \le
\frac12\left( \Robs+\Rsrcaff+\Tmedaff+\Rsurraff
\right).
\end{equation}
For fixed \(\Py\) and \(\Ps\), applying \eqref{eq:trace_trailing_ev_equality} to this bound shows that the leading generalized eigenspace of \((\Hmmaff(\Py,\Ps),\Rm)\) minimizes \(\Tmedaff\).
\end{proposition}

See \Cref{app:proof_affine_medium_control_kl} for the proof.

More generally, a matrix \(\Rm\) satisfying this Poincar\'e inequality need not be unique. Existence of such a metric is a genuine restriction on \(\priorm\) and requires sufficient concentration for the stated bound to hold. See \cite{zahm2020gradient,cui2021data} for sufficient conditions and examples beyond the Gaussian case.

The construction is therefore sequential rather than jointly optimal: first choose the output projector, then the source projector conditional on the output choice, and finally the medium projector conditional on the preceding two choices.

\section{SPPL Construction, Offline Algorithm, and Reduced-Space Inference}
\label{sec:practical_pipeline}

This section turns the preceding error decomposition into the SPPL construction. \Cref{tab:theory_to_algorithm} summarizes the sequential reduced-space choices and the associated training step.

\begin{table}[htbp]
\centering
\caption{Theory-to-algorithm map for the direct affine SPPL construction.
Each decomposition term motivates a selection criterion and corresponding construction step.}
\label{tab:theory_to_algorithm}
\small
\setlength{\tabcolsep}{4pt}
\renewcommand{\arraystretch}{1.05}
\begin{tabular}{@{}p{0.16\textwidth}
                    p{0.39\textwidth}
                    p{0.39\textwidth}@{}}
\toprule
\textbf{Error term}
&
\textbf{Selection criterion}
&
\textbf{Construction step}
\\
\midrule

\(\Robs\)
&
\(\Cy\)
&
Leading \(r_y\) output POD modes
\\

\(\Rsrcaff\)
&
\((\Hss(\Py),\Rs)\), given \(\Py\)
&
Top \(r_s\) source modes
\\

\(\Cmedaff\) (via \(\Tmedaff\))
&
\((\Hmmaff(\Py,\Ps),\Rm)\), given \(\Py,\Ps\)
&
Top \(r_m\) medium modes
\\

\(\Rsurraff\)
&
coefficient regression loss (defined below)
&
Train coefficient model
\\

\bottomrule
\end{tabular}
\end{table}

\subsection{SPPL architecture and coefficient-matrix targets}
\label{sec:sppl_architecture}

After the three bases are fixed, let \(\bmu=\Phim^*(\bvecm-\barm)\) and \(\balpha=\Phis^*(\bvecs-\bars)\). We define the SPPL surrogate as
\begin{equation}
\label{eq:sppl_affine_architecture}
\tilde{\Fwhite}_{\mathrm{aff}}(\Ptheta\btheta)
= \Phiy\Bigl(
\qsurr(\bmu;\bm{w})+
\Qsurr(\bmu;\bm{w})\,\balpha
\Bigr).
\end{equation}
We parameterize \((\qsurr,\Qsurr)\) by a network \(\mathcal N_{\bm w}\). The network receives only \(\bmu\in\mathbb R^{r_m}\) and returns
\( [\qsurr(\bmu),\Qsurr(\bmu)] \in\mathbb R^{r_y\times(r_s+1)}. \)
The retained source coordinates \(\balpha\) enter analytically outside the network through the affine combination in \eqref{eq:sppl_affine_architecture}. Thus the trainable map models only medium dependence, while affine dependence on the retained source coordinates is enforced by construction.

The corresponding computable medium-only coefficient targets are
\begin{equation}
\label{eq:affine_targets_medium_only}
\qproj(\bvecm)
\coloneqq
\Phiy^T\Gm\bars,
\qquad
\Qproj(\bvecm)
\coloneqq
\Phiy^T\Gm\Phis.
\end{equation}
Thus \(\qproj\) is the projected response to the source-prior mean, while the columns of \(\Qproj\) are the projected responses to the retained source basis directions. These are the medium-only regression targets used to train the SPPL network. Although these targets are evaluated at the full medium realization \(\bvecm\), the network receives only its retained coordinate \(\bmu\). It is therefore not expected to reproduce \(\qproj(\bvecm)\) and \(\Qproj(\bvecm)\) pointwise for all medium realizations sharing the same \(\bmu\); under the squared coefficient loss introduced below, the population-optimal predictors are instead their conditional expectations given \(\bmu\).

\subsection{Coefficient loss and population identity}
\label{sec:coefficient_targets_training}

To learn from these full-medium coefficient targets using only the retained coordinate \(\bmu\), we define the following loss for a trainable pair \((\qsurr,\Qsurr)\):
\begin{equation*}
\mathcal{L}_{\mathrm{coeff}}(\bm{w})
\coloneqq
\Ex{\bvecm\sim\priorm}{
\left\|\qsurr(\bmu;\bm{w})-\qproj(\bvecm)\right\|^2
+
\left\|\Qsurr(\bmu;\bm{w})-\Qproj(\bvecm)\right\|_{\mathrm{F}}^2 }.
\end{equation*}

Define the conditional mean source-to-observable map by
\[
\Gopt_r(\bmu)
\coloneqq
\E{\G(\bvecm)\mid\Phim^*(\bvecm-\barm)=\bmu}.
\]
The following result shows that the regression automatically averages over the discarded medium directions and separates this unavoidable medium-reduction error from trainable surrogate error.

\begin{proposition}[Coefficient-loss identity]
\label{prop:coefficient_loss_identity}
Define the optimal reduced coefficient functions
\begin{equation*}
\qsurr^*(\bmu)
\coloneqq
\Phiy^T\Gopt_r(\bmu)\bars,
\qquad
\Qsurr^*(\bmu)
\coloneqq
\Phiy^T\Gopt_r(\bmu)\Phis.
\end{equation*}
For the SPPL surrogate
\[
\tilde{\Fwhite}_{\mathrm{aff}}(\Ptheta\btheta)
=
\Phiy\bigl(\qsurr(\bmu;\bm w)+\Qsurr(\bmu;\bm w)\balpha\bigr),
\]
the expected coefficient loss satisfies
\begin{equation}
\label{eq:coefficient_loss_decomposition}
\mathcal{L}_{\mathrm{coeff}}(\bm w)
=
\Rsurraff(\bm w)+\Cmedaff,
\end{equation}
where, equivalently to \eqref{eq:Cmed_aff_equiv}, the same medium residual can be written as
\begin{equation}
\label{eq:Cmed_aff_def}
\Cmedaff
=
\Ex{\bvecm\sim\priorm}{
\left\|\Phiy^T\left(\Gm-\Gopt(\bvecm_r)\right)\bars\right\|^2
+
\left\|\Phiy^T\left(\Gm-\Gopt(\bvecm_r)\right)\Phis\right\|_{\mathrm F}^2 }.
\end{equation}
Since \(\Cmedaff\) is independent of \(\bm w\), minimizing \(\mathcal L_{\mathrm{coeff}}\) is equivalent to minimizing \(\Rsurraff\).
\end{proposition}

See \Cref{app:proof_expected_coeff_loss} for the proof.

\subsection{Empirical coefficient training}

Given medium samples \(\bvecm_1,\dots,\bvecm_{N_m}\sim\priorm\), we use the empirical loss
\begin{equation}
\label{eq:empirical_coefficient_loss}
\widehat{\mathcal{L}}_{\mathrm{coeff}}(\bm{w})
=
\frac{1}{N_m}
\sum_{j=1}^{N_m}
\left(
\left\|\qsurr(\bmu_j;\bm{w})-\qproj_j\right\|^2
+
\left\|\Qsurr(\bmu_j;\bm{w})-\Qproj_j\right\|_{\mathrm{F}}^2
\right),
\end{equation}
with
\begin{equation*}
\bmu_j=\Phim^*(\bvecm_j-\barm),
\qquad
\qproj_j=\Phiy^T\G(\bvecm_j)\bars,
\qquad
\Qproj_j=\Phiy^T\G(\bvecm_j)\Phis.
\end{equation*}
The unweighted Frobenius loss is appropriate because \(\E{\balpha}=0\) and \(\E{\balpha\balpha^T}=I_{r_s}\). The same loss applies in the Gaussian source case, where \Cref{cor:gaussian_affine_restriction_exactness} identifies the optimal affine-in-source approximation with the optimal blockwise ridge approximation. This is the empirical objective used to train the coefficient model. We next show how the PDE solves needed to assemble both the medium diagnostic and these coefficient targets can be reused efficiently.

\subsection{PDE-operator factorization reuse for the medium diagnostic and coefficient targets}
\label{sec:factorization_reuse}

Having defined the coefficient targets and both the population and empirical training objectives, we now show how the quantities needed for medium-space construction and coefficient training can be assembled efficiently. The medium-reduction step requires derivatives of the projected response to the source mean and retained source directions, while coefficient training requires the corresponding projected response values themselves. Both sets of quantities can be computed by reusing a single PDE-operator factorization for multiple forward and adjoint right-hand sides.

Let \(\bm\phi_i\), \(i=1,\ldots,r_s\), denote the columns of \(\Phis\), set \(\bm\phi_0\coloneqq\bars\), and let \(\bm\phi_{y,k}\), \(k=1,\ldots,r_y\), denote the columns of \(\Phiy\). Define the projected response coefficients
\begin{equation*}
 c_0(\bvecm)\coloneqq \Phiy^T\Gm\bars,
 \qquad
 c_i(\bvecm)\coloneqq \Phiy^T\Gm\bm\phi_i,
 \quad i=1,\ldots,r_s.
\end{equation*}
Because \(\Py=\Phiy\Phiy^T\), the affine medium diagnostic matrix from \eqref{eq:Hmmaff_moment_expansion} can be written directly as
\begin{equation}
\label{eq:Hmmaff_projected_coefficient_jacobians}
\Hmmaff(\Py,\Ps)
=
\Ex{\bvecm\sim\priorm}{
(\grad{\bvecm}c_0)^T(\grad{\bvecm}c_0)
+
\sum_{i=1}^{r_s}
(\grad{\bvecm}c_i)^T(\grad{\bvecm}c_i)
}.
\end{equation}
Thus the projected coefficient Jacobians provide exactly the derivative information needed for the medium reduced space.

For fixed \(\bvecm\), let \(\bm x_i\), \(i=0,\ldots,r_s\), and \(\bm p_k\), \(k=1,\ldots,r_y\), be the forward and adjoint states defined by
\begin{equation*}
\mathcal{L}(\bvecm)\bm x_i=\mathcal S\bm\phi_i,
\qquad i=0,\ldots,r_s,
\qquad
\mathcal{L}(\bvecm)^*\bm p_k
=
\mathcal B^*\,(\Signoise^{-1/2})^T\bm\phi_{y,k}
=
\mathcal B^*\,\Signoise^{-1/2}\bm\phi_{y,k}.
\end{equation*}
Since \(\mathcal S\) is independent of \(\bvecm\), the components of the projected coefficient Jacobians satisfy
\begin{align}
\label{eq:c0_medium_gradient}
\grad{\bvecm}(c_0)_k
&=
-\bm p_k^T
\left(
\grad{\bvecm}\mathcal L(\bvecm)
\right)\bm x_0,
\\
\label{eq:ci_medium_gradient}
\grad{\bvecm}(c_i)_k
&=
-\bm p_k^T
\left(
\grad{\bvecm}\mathcal L(\bvecm)
\right)\bm x_i,
\qquad i=1,\ldots,r_s.
\end{align}
Thus, for a general source prior, assembling \(\Hmmaff\) requires \(r_s+1\) forward solves and \(r_y\) adjoint solves per medium realization, all reusing a single factorization of \(\mathcal L(\bvecm)\). If \(\bars=0\), then \(\bm x_0=0\), \(c_0=0\), and \(\nabla_{\bvecm}c_0=0\), so the mean-source solve is omitted and the count reduces to \(r_s\) forward and \(r_y\) adjoint solves. The same forward solves provide the reduced coefficient values in \eqref{eq:affine_targets_medium_only}: \(\qproj=c_0\), while the \(i\)-th column of \(\Qproj\) is \(c_i\). Thus the same factorization-reuse workflow supplies both the derivative information needed to construct the medium space and the coefficient targets needed for training. This completes the offline construction and training workflow; we next describe reduced-space posterior inference.

\subsection{Reduced-space approximate posterior and inference for Gaussian priors}
\label{sec:sppl_reduced_posterior_sampling}

When the medium and source priors are Gaussian, let
\(\priorm=\N{\barm}{\bm C_m}\) and
\(\priors=\N{\bars}{\bm C_s}\), and use the canonical choices
\(\Rm=\bm C_m^{-1}\) and \(\Rs=\bm C_s^{-1}\). Then the retained coordinates
\[
\bmu=\Phim^*(\bvecm-\barm),
\qquad
\balpha=\Phis^*(\bvecs-\bars)
\]
are independent standard Gaussian vectors under the prior. Because the SPPL likelihood depends only on these retained coordinates, posterior inference can be performed directly in \((\bmu,\balpha)\). We define the resulting reduced-space approximate posterior by
\begin{equation}
\label{eq:sppl_reduced_posterior}
\pi_{\mathrm{SPPL},r}^{\yobs}(\bmu,\balpha)
\propto
\exp\!\left[
-\frac12
\left\|
\Phiy\bigl(\qsurr(\bmu)+\Qsurr(\bmu)\balpha\bigr)-\yobs
\right\|^2
\right]
\exp\!\left[
-\frac12\|\bmu\|^2-\frac12\|\balpha\|^2
\right].
\end{equation}
When physical-space posterior summaries are required, the discarded components can be reconstructed by drawing them from their conditional Gaussian prior given the retained coordinates. Thus online MCMC requires only evaluations of the reduced coefficient model and the analytic affine source action. The coseismic study in \Cref{sec:coseismic_protocol} uses this reduced-space formulation.

\section{Numerical Experiments}
\label{sec:numerical_experiments}
We consider numerical experiments that aim to address four questions. First, do the blockwise diagnostic matrices yield competitive projected approximations at matched retained dimension? Second, do the direct affine criteria remain effective for a non-Gaussian source prior? Third, on fixed block-reduced coordinates, what gains arise from coefficient-matrix supervision and from enforcing affine source dependence in the surrogate architecture? Finally, how faithfully do the resulting surrogates reproduce the posterior, and how robust are they to changes in source scale? 
To this end, we present two sets of numerical experiments: quadratic benchmarks that address the first two questions and a coseismic fault-slip inverse problem that addresses the latter two.

\subsection{Quadratic benchmarks}
\label{sec:toy_validation}
First, we consider an analytic forward map that allows us to compare in detail the blockwise and joint reduced-coordinate constructions at matching retained dimensions 
along with the direct affine errors and their medium spectral bounds in the non-Gaussian setting.

Here, the forward map takes the following form:
\begin{equation}
\label{eq:quadratic_benchmark_map}
\begin{aligned}
\Fwhite(\bvecm,\bvecs)
&=\G(\bvecm)\bvecs,\\
\G(\bvecm)
&=
\mathbf A_0
+\sum_{i=1}^{d_m}m_i\mathbf A_i
+\sum_{i,j=1}^{d_m}m_i m_j\mathbf B_{ij},
\qquad
\mathbf B_{ij}=\mathbf B_{ji},
\end{aligned}
\end{equation}
which is linear in the source and quadratic in the medium.
For this map, the required conditional moments and blockwise affine error terms are available analytically. 
We consider choices of the coefficients $\mathbf{A}_0, \mathbf{A}_i, \mathbf{B}_{ij}$
that correspond to regimes with different contributions of the source and medium inputs. 
The coefficient-generation procedure and parameter regimes are given in \Cref{app:quadratic_benchmark_details}.

We compare four matched basis constructions that form a \(2\times2\) comparison: block-separable versus unrestricted joint reduction, and prior-mean versus prior-averaged derivative information. 
All experiments use \(d_m=d_s=n_{\mathrm{obs}}=10\) and no output reduction, so \(\Py=I\). 
Although \(\Py=I\) in these benchmarks, we retain \(\Py\) in the formulas below to match the theory notation. Recall that \(J_{\bvecs}(\bvecm)=\nabla_{\bvecm}(\Gm\bvecs)\) denotes the medium Jacobian evaluated at the current source. Let
\(\bar{\btheta}=(\barm,\bars)\),
\[
J_{\btheta}(\bvecm,\bvecs)
\coloneqq
\grad{\btheta}\Fwhite(\bvecm,\bvecs)
=
\begin{bmatrix}J_{\bvecs}(\bvecm)&\Gm\end{bmatrix},
\qquad
\bm R_{\btheta}=\diag(\Rm,\Rs).
\]
We compare diagnostic matrices built either from a prior-mean Gauss--Newton Hessian (PM-GNH) or from data-free likelihood-informed-subspace (DF-LIS) averaging:
\begin{align*}
\mathbf H_{ss}^{\mathrm{PM\text{-}GNH}}
&=\G(\barm)^T\Py\G(\barm),
&
\mathbf H_{mm}^{\mathrm{PM\text{-}GNH}}
&=J_{\bars}(\barm)^T\Py J_{\bars}(\barm),
\\
\mathbf H_{ss}^{\mathrm{DF\text{-}LIS}}
&=\Hss(\Py),
&
\mathbf H_{mm}^{\mathrm{DF\text{-}LIS}}
&=\Hmmaff(\Py,\Ps),
\\
\mathbf H_{\btheta}^{\mathrm{PM\text{-}GNH}}
&=J_{\btheta}(\barm,\bars)^T\Py J_{\btheta}(\barm,\bars),
&
\mathbf H_{\btheta}^{\mathrm{DF\text{-}LIS}}
&=\Ex{(\bvecm,\bvecs)\sim\prior}{J_{\btheta}^T\Py J_{\btheta}}.
\end{align*}
Block prior-mean GNH and Block DF-LIS use the dominant generalized eigenvectors of the
corresponding source and medium pairs \((\mathbf H_{ss},\Rs)\) and
\((\mathbf H_{mm},\Rm)\), whereas Joint prior-mean GNH and Joint DF-LIS use the
dominant generalized eigenvectors of \((\mathbf H_{\btheta},\bm R_{\btheta})\).
Thus the blockwise constructions are block-separable, while the joint constructions may
mix medium and source directions.

To compare blockwise and joint constructions under a common retained-dimension budget while allowing Block DF-LIS to determine the allocation between source and medium directions, we use the following rank-matching rule. 
We first request the source ranks to be \(r_s\in\{1,2,4,8\}\). For each \(r_s\), a cutoff
\(\tau\) is chosen from the Block DF-LIS source spectrum so that exactly \(r_s\)
source eigenvalues exceed \(\tau\); applying the same cutoff to the resulting
Block DF-LIS medium spectrum determines \(r_m\). Both blockwise methods use this common pair \((r_s,r_m)\). Joint DF-LIS uses the full matched dimension \(r_s+r_m\). For Joint prior-mean GNH, however, the diagnostic matrix is formed from the Jacobian evaluated at the prior mean. Since \(n_{\mathrm{obs}}=10\), this Jacobian has rank at most 10, and hence the corresponding Gauss--Newton Hessian has at most 10 nonzero generalized eigenvalues. We therefore cap the retained joint rank at \(\min(r_s+r_m,10)\), rather than adding directions from its nullspace. Consequently, once the requested matched dimension \(r_s+r_m\) reaches 10, the Joint prior-mean GNH subspace no longer changes, and its normalized total reduction error remains constant as the requested dimension is increased further.

Throughout the quadratic benchmarks, the normalized total error of an
approximation \(\tilde{\Fwhite}\) is the prior-averaged squared forward error
\(\Ex{\prior}{\|\Fwhite-\tilde{\Fwhite}\|^2}\) normalized by
\(\Ex{\prior}{\|\Fwhite\|^2}\); decomposition terms use the same normalization.
The coefficient-generation procedure, regime parameters, exact conditional
formulas, and Monte Carlo evaluation details are given in
\Cref{app:quadratic_benchmark_details,app:toy_gmm_joint_benchmark}.

\subsubsection{Gaussian source prior: exactness of the affine restriction}
\label{sec:toy_gaussian}

We consider source-dominated, approximately balanced, and medium-dominated
regimes, generated by varying the source-mean amplitude and the relative strength of the constant operator component; \Cref{app:quadratic_benchmark_details} gives the corresponding parameter choices. For Gaussian sources, the affine restriction introduces no additional error within the retained block
spaces, so the reported error isolates reduction by the selected bases.
\begin{figure}[htbp]
    \centering
    \includegraphics[width=\textwidth]{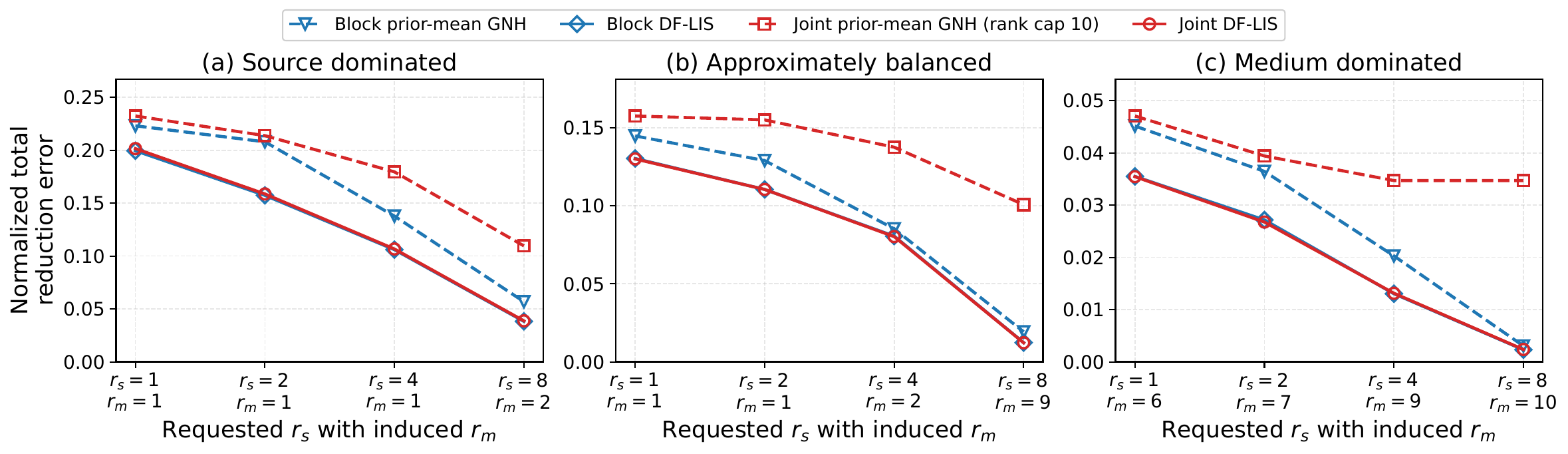}
    \caption{Gaussian-source quadratic benchmark across three regimes: source-dominated, approximately balanced, and medium-dominated. Each panel reports normalized total reduction error versus requested source rank \(r_s\) for Block prior-mean GNH, Block DF-LIS, Joint prior-mean GNH, and Joint DF-LIS at common requested dimension budgets. Joint prior-mean GNH is capped at rank 10.}
    \label{fig:toy_gaussian_three_regimes}
\end{figure}

Figure~\ref{fig:toy_gaussian_three_regimes} shows the truncation errors for the different basis constructions. It is observed that Block DF-LIS closely matches Joint DF-LIS at matched retained dimension across all three regimes, while both prior-mean GNH constructions incur consistently larger error. Thus block-separable reduction loses little relative to unrestricted joint DF-LIS here, whereas prior averaging of derivative information is important for basis quality.

\subsubsection{Gaussian-mixture source prior: direct affine error and spectral control}
\label{sec:toy_gmm}

We next repeat the three regimes using a two-component Gaussian-mixture source prior; \Cref{app:quadratic_benchmark_details} gives the construction and parameters. For this non-Gaussian prior, the optimal retained-source approximation is not affine. Accordingly, the blockwise methods are evaluated using the optimal affine-in-source approximation on their selected block subspaces, whereas the joint methods are evaluated using the exact conditional expectation given their retained joint coordinates. Thus the joint methods are allowed a more general approximation than the blockwise methods.

Figure~\ref{fig:toy_gmm_three_regimes} shows that Block DF-LIS remains competitive with Joint DF-LIS under the Gaussian-mixture source prior. The computable source-plus-medium bound \(\Rsrcaff+\Tmedaff\) also tracks the exact contribution \(\Rsrcaff+\Cmedaff\) closely, supporting the direct affine criteria beyond the Gaussian setting.

\begin{figure}[htbp]
    \centering
    \includegraphics[width=\textwidth]{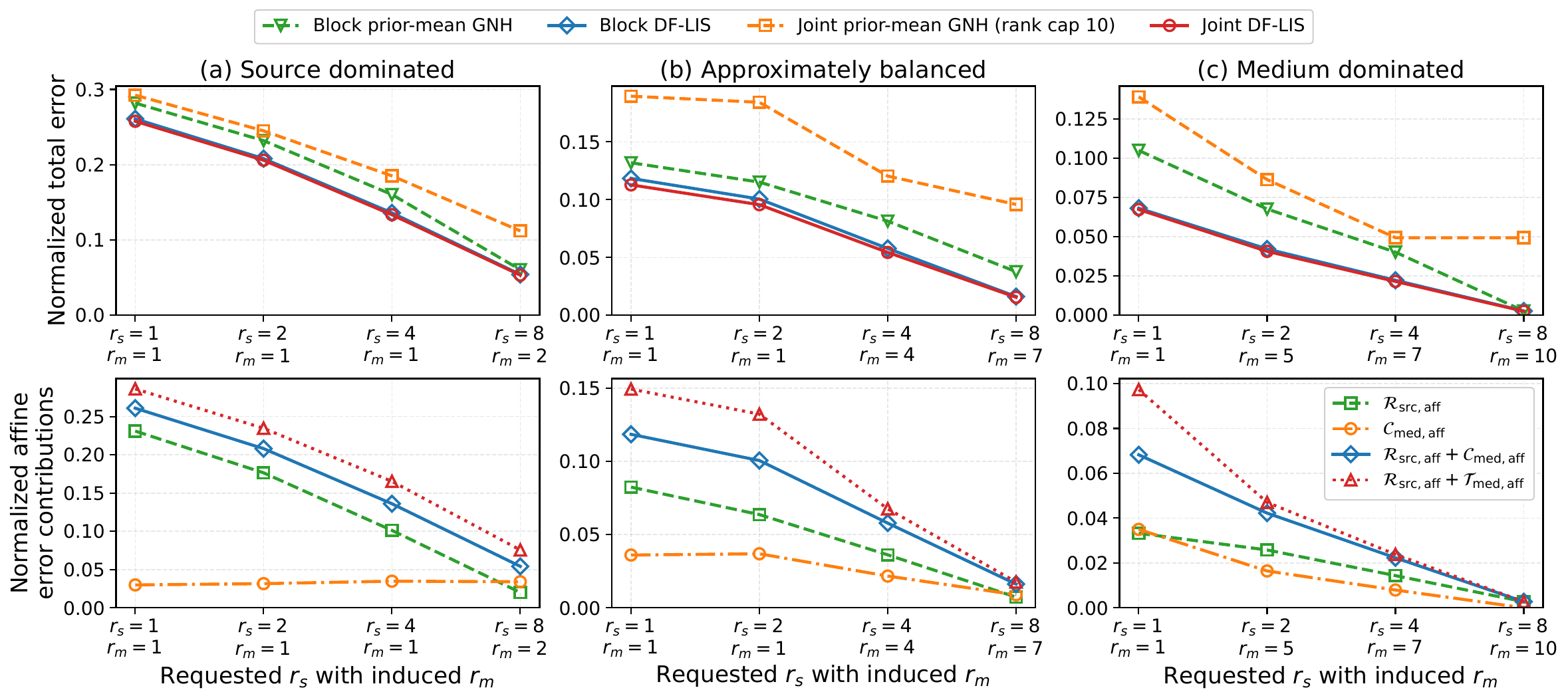}
    \caption{Gaussian-mixture-source quadratic benchmark across three regimes.
    Top: normalized total error for Block prior-mean GNH, Block DF-LIS, Joint prior-mean GNH,
    and Joint DF-LIS at matched retained dimension. Bottom: normalized
    \(\Rsrcaff\), \(\Cmedaff\), their sum, and the bound
    \(\Rsrcaff+\Tmedaff\) for Block DF-LIS.}
    \label{fig:toy_gmm_three_regimes}
\end{figure}

\subsection{Application study: surrogate supervision and source parameterization in coseismic fault-slip inversion}
\label{sec:coseismic_fault}

We now test the full SPPL pipeline in a PDE-constrained Bayesian inversion problem using a three-model surrogate ablation under a matched offline PDE budget. All three surrogates use the same block-reduced output, medium, and source coordinates. The study therefore focuses on two surrogate-design choices, coefficient-matrix versus source-probe supervision and structure-preserving affine-in-source versus generic nonlinear parameterization. 
The surrogate models are evaluated in terms of their prediction accuracy, posterior fidelity, and online sampling efficiency. 

We also evaluate the out-of-distribution performance of the surrogate-design choices.
This is particularly relevant to Bayesian inverse problems, as the posterior may concentrate in regions of the parameter space to which the prior assigns little mass, e.g., when the data are informative or when the prior is misspecified.
In such cases, surrogates trained using the prior distribution may yield poor approximations in posterior-relevant regions \cite{li2014adaptive_surrogates_bip,yan2020adaptive_dnn_bip}. 
However, by explicitly incorporating the affine source dependence, SPPL surrogates may be able to accommodate distributional shifts that occur in the source. 
To investigate this behavior, we compare the performance of the surrogate models in the context of a miss-specification in variance of the source prior distribution.

\subsubsection{Physical model and discretization}
\label{sec:coseismic_model}

We use the two-dimensional weakly symmetric Hellinger--Reissner elasticity formulation and fault treatment of \cite{Puel22}, together with the heterogeneous-material/fault-slip joint-inversion setting of \cite{puel2024adjoint_joint}, closely following the coseismic application setting of \cite{Puel_2024}. The medium field controls the shear modulus and hence the elasticity operator, while the fault-slip field prescribes the tangential displacement jump and enters as a linear source. Let \(\bvecm\) denote the finite-element coefficients of the medium parameter field \(m\), which parameterizes the spatially varying shear modulus, and let \(\bvecs\) denote the finite-element coefficients of the fault-slip field \(s\). After finite-element discretization, the mixed state equation and whitened parameter-to-observable map have the form
\[
\mathbf K(\bvecm)\bm x=\mathbf S\bvecs,
\qquad
\Fwhite(\bvecm,\bvecs)
=
\Signoise^{-1/2}\mathcal B\mathbf K(\bvecm)^{-1}\mathbf S\bvecs
=
\G(\bvecm)\bvecs,
\]
where \(\mathbf K(\bvecm)\) is the medium-dependent discrete elasticity operator and \(\mathbf S\) is the fixed operator that injects fault slip into the state equation. Thus the forward map is nonlinear in the medium parameters \(\bvecm\) and linear in the fault-slip parameters \(\bvecs\), giving the partially linear structure studied above.

The domain is a \(500\times250\) km vertical cross-section with a dipping fault and 20 two-component surface GPS stations, so \(n_{\mathrm{obs}}=40\). The discretization has state dimension \(145083\), medium dimension \(d_m=3101\), and slip dimension \(d_s=29\), and the observation-noise standard deviation is \(0.015\). The implementation uses hIPPYlib \cite{villa2018hippylib,villa2021hippylib_toms}. We use independent zero-mean Gaussian priors, taking the medium prior precision as the canonical Poincar\'e metric matrix \(\Rm\) and the source prior precision matrix as \(\Rs\). Thus, \(\barm=\bars=0\).

Noise-free synthetic observations are generated by a forward solve at the prescribed truth \((m^\dagger,s^\dagger)\); the likelihood nevertheless uses the prescribed nonzero observation-noise covariance. We denote the resulting nominal whitened data by \(\yobs_1=\Fwhite(m^\dagger,s^\dagger)\). The full elasticity equations, fault conditions, compliance law, discretization, priors, truth fields, and synthetic-data construction are collected in \Cref{app:coseismic_details}.

\subsubsection{Three-model ablation: separating source supervision from surrogate parameterization}
\label{sec:coseismic_protocol}

We compare the performance of our proposed SPPL surrogate with alternative surrogate architectures that use the same reduced spaces to highlight the effect of explicitly preserving the affine source dependence. To this end, we consider three approaches: DIPNet, Matrix-loss SPPL, and Probe-loss SPPL.

Here, DIPNet~\cite{oleary2022derivative} serves as a generic projected neural-surrogate baseline. In the common retained coordinates, DIPNet takes both parameter blocks as network inputs and directly predicts the reduced observable,
\[
\mathcal N_{\mathrm{DIPNet}}:\mathbb R^{r_m+r_s}\to\mathbb R^{r_y},
\qquad
\tilde{\Fwhite}_{\mathrm{DIPNet}}(\bmu,\balpha)
=\Phiy\,\mathcal N_{\mathrm{DIPNet}}(\bmu,\balpha).
\]
Thus, unlike SPPL, DIPNet does not impose affine dependence on \(\balpha\); it learns the joint dependence on the retained medium and source coordinates from data. Here \(r_m=r_s=12\) and \(r_y=17\), so the network maps \(\mathbb R^{24}\) to \(\mathbb R^{17}\). For each medium realization \(\bvecm_j\), we draw \(r_s\) independent source probes \(\bvecs_{jk}\), form \(\balpha_{jk}=\Phis^*(\bvecs_{jk}-\bars)\), and compute the reduced-output targets \(\bm z_{jk}=\Phiy^T\Fwhite(\bvecm_j,\bvecs_{jk})\). DIPNet is trained on the pairs \((\bmu_j,\balpha_{jk})\mapsto\bm z_{jk}\) and must therefore learn the source dependence from sampled probes.

Matrix-loss SPPL is the coefficient-supervised model developed in \Cref{sec:coefficient_targets_training}. It learns the reduced 
affine offset \(\qproj(\bvecm)=\Phiy^T\Gm\bars\) and 
coefficient matrix \(\Qproj_j=\Phiy^T\G(\bvecm_j)\Phis\), preserving the linear source action by construction. Because this gives SPPL direct supervision of the complete map on the retained source space, a direct comparison with DIPNet would change both the architecture and the training supervision.

To address this, we introduce Probe-loss SPPL as an intermediate comparator. It uses the same affine-in-source architecture as Matrix-loss SPPL but is trained on exactly the same realized probe--output pairs as DIPNet. Thus Probe-loss SPPL versus DIPNet isolates the effect of the affine source parameterization, while Matrix-loss versus Probe-loss SPPL isolates the effect of direct coefficient-matrix supervision. Under the standing source-moment assumptions, the two SPPL losses have the same population minimizer; the empirical Probe-loss objective and derivation are given in \Cref{app:probe_loss_population}.

The ablation then compares DIPNet, Matrix-loss SPPL, and Probe-loss SPPL at training budgets \(N_m\in\{100,400,1600,6400\}\). 
At each budget, the models use the same medium realizations, the same output, medium, and source spaces, and the same number of PDE right-hand-side solves per medium realization. 
The block reduced spaces are constructed by the sequential SPPL criteria in \Cref{sec:medium_reduction_criteria}. Retaining modes with eigenvalue \(\lambda_i>1.0\) gives \(r_y=17\) and \(r_m=r_s=12\). 
The two SPPL variants share one fixed network architecture, while DIPNet uses the corresponding generic projected-network architecture; no extensive hyperparameter sweeps are performed.
Because the source prior is zero-mean, \(\bars=0\), the affine offset target in the SPPL model is \(\qproj(\bvecm)=\Phiy^T\Gm\bars=0\). Accordingly, the two SPPL models set \(\qsurr\equiv0\) and learn only the coefficient-matrix map \(\Qsurr\).

\paragraph{Posterior accuracy evaluation protocol}
For the Gaussian-prior coseismic experiments, surrogate chains sample their corresponding posteriors directly in the common retained coordinates \((\bmu,\balpha)\), with \eqref{eq:sppl_reduced_posterior} giving the SPPL case, using MAP-based likelihood-informed pCN (LA-pCN)~\cite{cui2016dimension,cotter2013mcmc}; the high-fidelity (HF) reference chains use the PDE model. Unlike the prior-predictive KL bound in the theory, the numerical posterior comparisons use the fixed-data forward KL divergence \(D_{\mathrm{KL}}(\post\,\|\,\tilde{\pi}_{\mathrm{post}}^{\yobs})\). For each datum, this quantity is estimated from samples of the corresponding HF posterior by evaluating the HF and surrogate posterior potentials at the same samples; the explicit estimator and convergence diagnostics are given in \Cref{app:coseismic_evaluation_protocol}.

\paragraph{Source-scale robustness evaluation protocol}

To test robustness to changes in the scaling of the source prior, 
we consider a set of inverse problems that are obtained by rescaling the true source and data, i.e.,
\[
c\in\left\{1/10,1/\sqrt{10},1,\sqrt{10},10\right\},
\qquad
s_c^\dagger=c\,s^\dagger,
\qquad
\yobs_c=c\,\yobs_1,
\]
with the true medium, priors, and likelihood noise covariance unchanged. 
This is analogous to rescaling prior variance for the source.
A separate HF chain targets \(\pi_{\mathrm{post}}^{\yobs_c}\) for each test scale \(c\) and provides reference samples from that posterior.
The surrogate models, whose training sets correspond to the source scaling $c_{\mathrm{train}} = 1$, are then evaluated using the fixed-data KL diagnostic on the posteriors with $c \in\left\{1/10,1/\sqrt{10},1,\sqrt{10},10\right\}$.
For additional comparisons, we also train DIPNets on rescaled training sets, i.e., \(c_{\mathrm{train}}\in\{1/10,1,10\}\),
as well as the union of all three scale-specific training sets.
This final combined model uses three times the probe data and is included only as a robustness comparator, not in the matched-cost three-model ablation. The forward-level source-scale diagnostic additionally spans all training budgets and includes both SPPL variants. All models use the same reduced coordinates; training details are given in \Cref{app:network_training_details}.

\paragraph{Online efficiency diagnostics}
Online efficiency is reported through forward-evaluation cost, wall-clock cost per LA-pCN step, and integrated autocorrelation time (IACT); convergence details are given in \Cref{app:coseismic_evaluation_protocol}.

\subsubsection{Numerical results}

\paragraph{Reference inversion and MAP reconstruction}
Figure~\ref{fig:map_visualization} shows (i) the physical shear-modulus field corresponding to the prescribed medium truth and the physical shear-modulus field corresponding to its high-fidelity MAP estimate, (ii) the true displacement field and surface observation locations, and (iii) the true and MAP fault-slip profiles.

\begin{figure}[htbp]
    \centering
    \includegraphics[width=0.9\linewidth]{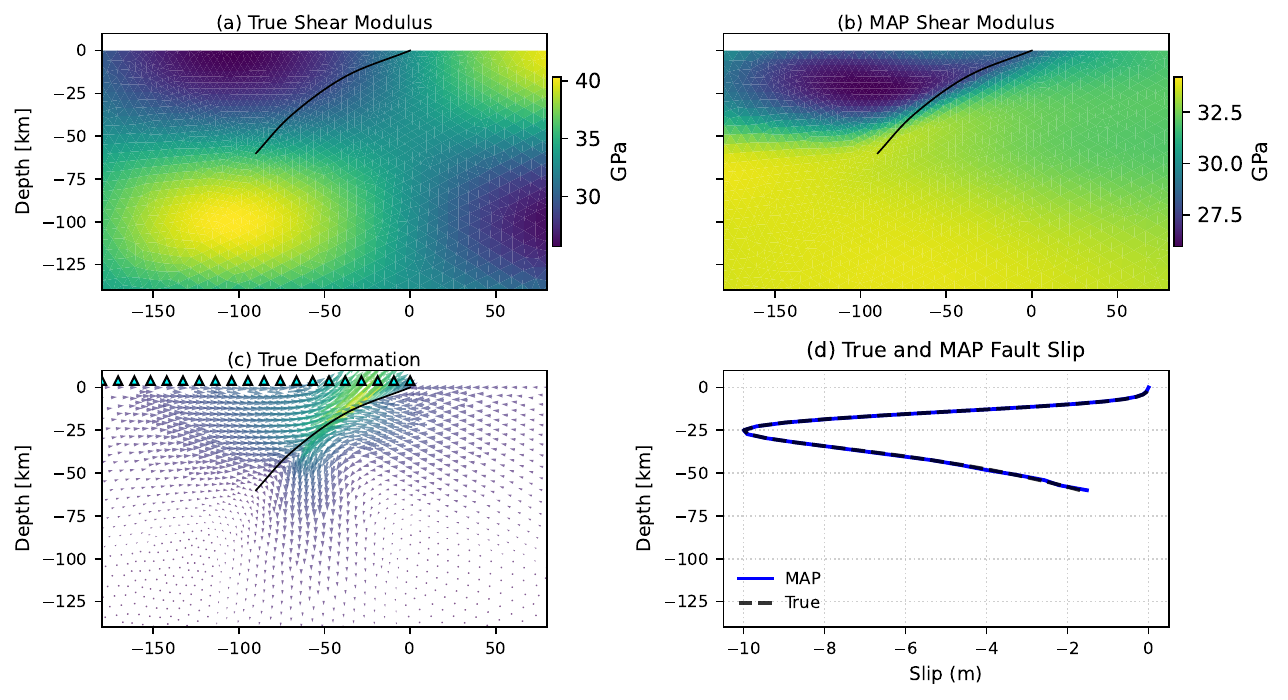}
    \caption{Coseismic inversion setup and high-fidelity MAP reconstruction.
    (a) True physical shear-modulus field \(\mu(m^\dagger)\).
    (b) Physical shear-modulus field corresponding to the high-fidelity MAP estimate, \(\mu(m_{\mathrm{MAP}})\).
    (c) True displacement field, with surface observation stations indicated.
    (d) True and MAP fault slip as functions of depth.
    Panels (a)--(c) show a zoomed region surrounding the fault rather than the full \(500\times250\) km computational domain.}
    \label{fig:map_visualization}
\end{figure}

The high-fidelity MAP estimate recovers the fault slip almost exactly. In contrast, the recovered physical shear-modulus field is smoother than the truth. Its spatial structure is captured most clearly in the shallower part of the domain, above the downdip end of the fault, whereas deeper variations remain poorly resolved.

\paragraph{Matched-cost forward error}
We assess surrogate accuracy through two complementary diagnostics in Figure~\ref{fig:forward_error_source_scale}: panel (a) measures prior-weighted forward accuracy as the offline training budget varies, whereas panel (b) measures posterior robustness to changes in source amplitude. We first consider the matched-cost comparison in panel (a), which reports the prior-averaged squared forward error \(\Ex{\prior}{\|\Fwhite-\tilde{\Fwhite}\|^2}\) under the matched offline PDE budget, with curves summarizing the three paired training seeds. By \eqref{eq:prior_predictive_kl_forward_bound}, this quantity directly controls the prior-predictive expected KL divergence between the high-fidelity and surrogate-induced posteriors.

\begin{figure}[htbp]
    \centering
    \includegraphics[width=0.9\linewidth]{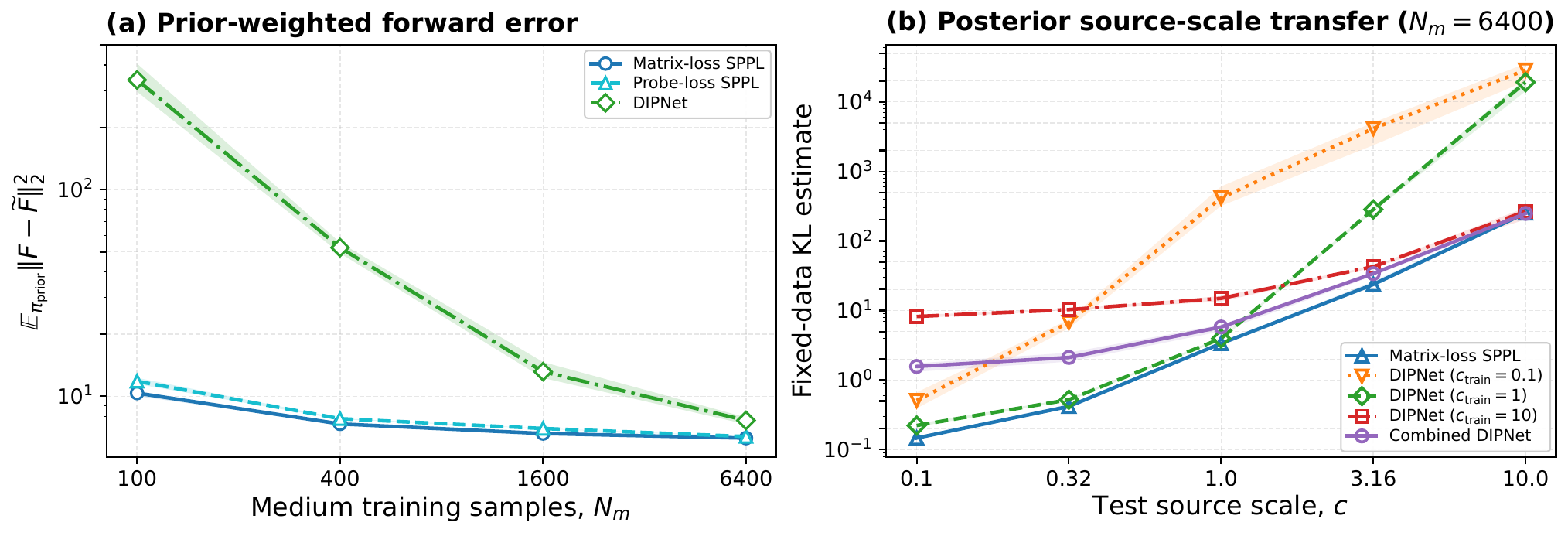}
    \caption{Two complementary surrogate diagnostics. (a) Matched-cost prior-averaged squared forward error versus \(N_m\) for Matrix-loss SPPL, Probe-loss SPPL, and DIPNet; curves summarize three paired training seeds. (b) Fixed-data posterior robustness under changes in source amplitude, measured by the posterior KL divergence versus the data scale \(c\) in \(\yobs_c=c\,\yobs_1\) at \(N_m=6400\). A separate high-fidelity posterior is sampled for each \(c\); Matrix-loss SPPL is compared with DIPNets trained at \(c_{\mathrm{train}}\in\{1/10,1,10\}\) and with the Combined DIPNet trained on all three scales.}
    \label{fig:forward_error_source_scale}
\end{figure}

Figure~\ref{fig:forward_error_source_scale}(a) shows that both SPPL variants are substantially more data-efficient than DIPNet at smaller training budgets, although the gap narrows as more training data are added. Matrix-loss and Probe-loss SPPL are nearly indistinguishable across the tested budgets. Because Probe-loss SPPL uses exactly the same realized probe--output pairs as DIPNet, the results suggest that the improvement arises primarily from the structure-preserving affine source parameterization rather than from coefficient-matrix supervision. The agreement between the two SPPL variants also shows that removing the finite-probe sampling error provides little additional benefit in this experiment.

 \paragraph{Source-scale robustness}
Figure~\ref{fig:forward_error_source_scale}(b) reports the fixed-data posterior KL divergence across the five data scales defined above. Each scale-specific DIPNet is most accurate near the source amplitudes represented during its training and deteriorates as the test-data scale moves away from that regime. Matrix-loss SPPL remains accurate across the range of \(c\), while the Combined DIPNet achieves comparable robustness only after being trained on all three source-amplitude regimes. This demonstrates the benefit of preserving the known linear source dependence rather than learning its behavior across source amplitudes from additional training data.

Figure~\ref{fig:ood_rmse} complements the posterior diagnostic with a forward-level view, reporting the relative observation error \(\|\tilde{\Fwhite}-\Fwhite\|/\|\Fwhite\|\) against the normalized retained-source radius \(\rho_s=\|\balpha\|_2/\sqrt{12}\) for the four training budgets. The test set consists of 1,000 independent medium--source pairs, with media drawn from the medium prior and retained-source coordinates drawn from \(\mathcal N(0,I_{12})\); each source is then rescaled to the prescribed radii, and its high-fidelity PDE observation is scaled accordingly using source linearity.
Across all four budgets, both SPPL variants maintain low error over essentially the full source-radius range, even at low training budgets with \(N_m=100\). On the other hand, while increasing \(N_m\) reduces DIPNet error and broadens its low-error region, its accuracy remains tied to the source amplitudes represented during training. Because all models use identical retained spaces, these differences reflect surrogate parameterization rather than dimension reduction,
with SPPL models remaining accurate for sources that lie far outside of the training distribution.

\begin{figure}[htbp]
    \centering
    \includegraphics[width=1\linewidth]{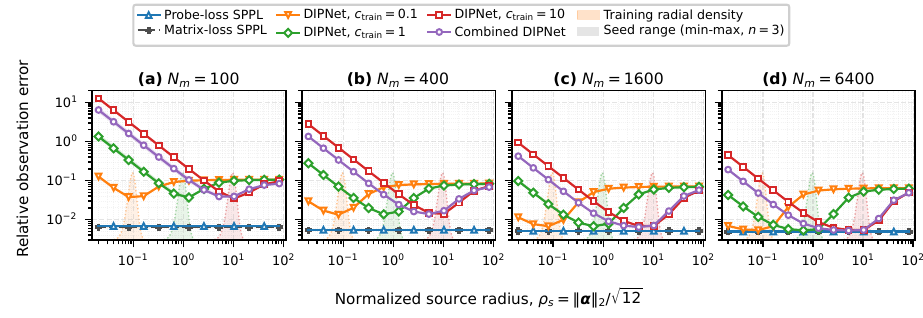}
    \caption{Relative observation error versus normalized source radius
    \(\rho_s=\|\balpha\|_2/\sqrt{12}\)
    for \(N_m=100\), \(400\), \(1600\), and \(6400\), with one panel per
    training budget. In every panel, DIPNet is trained separately at
    \(c_{\mathrm{train}}\in\{1/10,1,10\}\) and evaluated across the displayed
    source-amplitude range; Matrix-loss and Probe-loss SPPL are shown as common
    structure-preserving references. All models use the same block-reduced
    coordinates. Curves show the mean over three paired training seeds, and the
    seed envelopes show the minimum--maximum range across the three seeds
    (\(n=3\)).}
    \label{fig:ood_rmse}
\end{figure}

\paragraph{Posterior alignment}
Having established the matched-cost forward accuracy and source-scale behavior, we close the accuracy analysis with a more detailed examination of the nominal (\(c=1\)) posterior. Figure~\ref{fig:posterior_alignment} compares the Matrix-loss SPPL surrogate posterior trained with \(N_m=6400\) medium realizations against the high-fidelity PDE-based posterior in both reduced and physical coordinates.
The reduced-coordinate summaries agree closely, and the physical fault-slip posterior mean, credible interval, and MAP estimate are nearly indistinguishable. Thus the surrogate approximation does not materially affect the recovered posterior quantities of interest.

\begin{figure}[htbp]
    \centering
    \includegraphics[width=\linewidth]{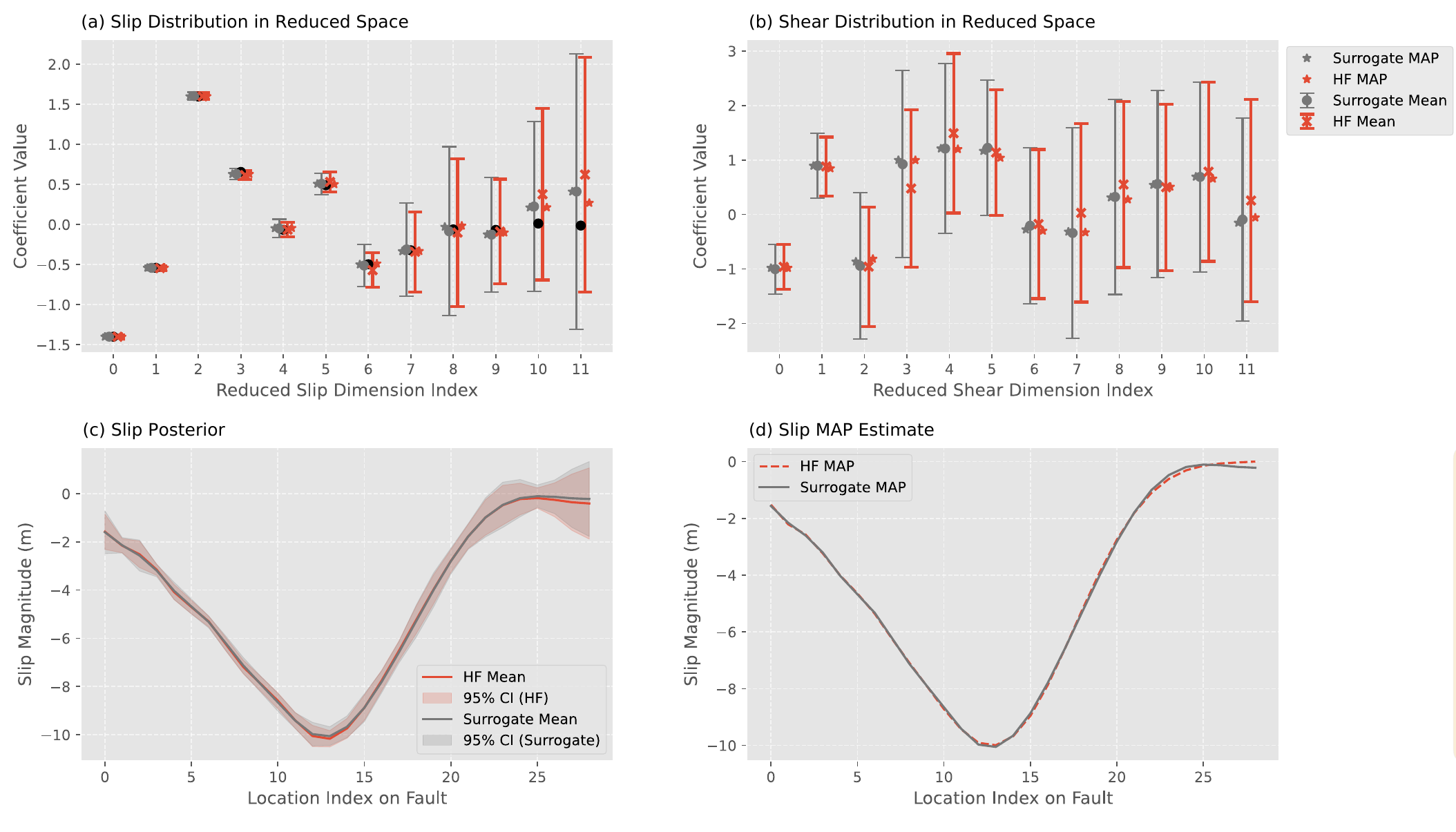}
    \caption{Alignment between the Matrix-loss SPPL surrogate posterior trained with \(N_m=6400\) medium realizations and the high-fidelity (HF) PDE-based posterior.
(a) Reduced slip-coordinate posterior summaries and
(b) reduced shear-modulus-coordinate posterior summaries.
Gray denotes the Matrix-loss SPPL posterior and red denotes the HF posterior. Stars denote the corresponding MAP estimates, markers denote posterior means, and error bars show the respective 95\% posterior credible intervals.
(c) Physical-coordinate fault-slip posterior mean and 95\% posterior credible interval for the HF and surrogate posteriors.
(d) Physical-coordinate fault-slip MAP estimates from the HF and surrogate posteriors.}
    \label{fig:posterior_alignment}
\end{figure}

\paragraph{Computational efficiency and offline break-even}
At the largest offline budget, \(N_m=6400\), the surrogate approaches provide several orders of magnitude of online acceleration relative to the high-fidelity PDE model. All wall-clock timings reported here were measured on a single core of an Intel Xeon CPU MAX 9480 processor on a Stampede3 Sapphire Rapids (SPR) compute node at TACC, with no MPI parallelism. The wall-clock time of a single forward evaluation decreases from approximately \(2.5\,\mathrm{s}\) to below \(5\times10^{-4}\,\mathrm{s}\), while the time per LA-pCN step decreases from approximately \(4\,\mathrm{s}\) to below \(2\times10^{-3}\,\mathrm{s}\). Sampling efficiency also improves substantially: the mean integrated autocorrelation time (IACT) decreases from \(22.8\) for the high-fidelity chain to below \(7\) for the surrogate-based chains. This improvement is consistent with performing MCMC only in the retained reduced coordinates. Taken together, these timings indicate an online MCMC speedup of at least a factor of $10^3$.

These online gains should be interpreted together with the offline cost required to construct the reduced spaces and train the surrogates. In this Gaussian-source application, \(\Cy\) and \(\Hss(\Py)\) are estimated from 200 common medium-prior samples, while \(\Hmmaff(\Py,\Ps)\) is estimated from 200 additional independent medium-prior samples; estimation of the affine medium diagnostic matrix accounts for the dominant reduction-stage expense. The adjoint-reuse strategy from \Cref{sec:factorization_reuse} is therefore essential for keeping this stage practical. For the largest Matrix-loss SPPL model with \(N_m=6400\), the total offline wall-clock time for PDE evaluations was approximately \(5.8\) hours, comparable to the time required to generate about \(200\) effective samples from the high-fidelity posterior. Network training cost remained negligible relative to PDE cost. Thus the reported break-even comparison places the online speedup in the context of the one-time offline investment needed to obtain it.

\section{Conclusion}
\label{sec:conclusion}

We have presented the structure-preserving partially linear (SPPL) framework for dimension reduction and surrogate modeling in Bayesian inverse problems with nonlinear medium dependence and linear source dependence. The theory characterizes the optimal blockwise ridge and affine-in-source approximations, establishes separate spectral criteria for output, source, and medium reduction, and provides prior-predictive posterior-error control. For Gaussian source priors, the optimal affine-in-source approximation coincides with the optimal blockwise ridge approximation. Quadratic benchmarks demonstrate that the blockwise reduction remains competitive with unrestricted joint reduction, including for non-Gaussian source priors. In the coseismic fault-slip inversion, a controlled three-model ablation indicates that the principal improvement in surrogate accuracy and data efficiency arises from preserving the known linear source dependence rather than from coefficient-matrix supervision alone. This structure also provides robustness to changes in source amplitude and enables orders-of-magnitude acceleration of online posterior sampling.

The present framework assumes a known partially linear decomposition and independent medium and source priors, while the PDE-based numerical validation is limited to a synthetic two-dimensional coseismic problem. Several directions merit further investigation. First, extending the optimal approximation and dimension-reduction theory to multilinear and polynomial parameter dependence, as well as identifying linear and nonlinear parameter directions in more general nonlinear models, would broaden the applicability of structure-preserving reduction. Second, for Gaussian source priors, analytically marginalizing the source variables would enable MCMC directly on the medium posterior, potentially improving sampling efficiency and motivating new dimension-reduction theory for the marginalized problem. Finally, an important direction is to extend SPPL to large-scale, three-dimensional geophysical inverse problems, particularly the joint Bayesian inversion of shear modulus and fault slip for the 2011 Tohoku-Oki earthquake, building on the inversion setting of Puel et al.~\cite{Puel_2024}. These developments offer a promising path toward efficient, structure-preserving Bayesian inference for increasingly complex and realistic PDE-constrained inverse problems.

\appendix
\crefalias{section}{appendix}

The appendices below collect proofs, quadratic benchmark formulas, and application reproducibility details used in the main text.

\section{Auxiliary projection and source-moment identities}
\label{app:auxiliary_projection_source_identities}

\begin{lemma}[Projection and source-moment identities]
\label{lem:projection_source_identities}
Let \(X\in L^2\) be vector-valued, let \(\mathcal G\) be a sub-\(\sigma\)-algebra,
and let \(Y\in L^2\) be \(\mathcal G\)-measurable. Then
\begin{equation}
\label{eq:conditional_pythagoras_appendix}
\E{\|X-Y\|^2}
=
\E{\|X-\E{X\mid\mathcal G}\|^2}
+
\E{\|\E{X\mid\mathcal G}-Y\|^2}.
\end{equation}
Moreover, with
\[
\balpha=\Phis^T\Rs(\bvecs-\bars),
\qquad
\widehat{\bvecs}=\bars+\Phis\balpha,
\qquad
 e_s=\bvecs-\widehat{\bvecs},
\]
we have
\begin{equation}
\label{eq:source_moment_identities_appendix}
\begin{aligned}
\E{\balpha} &= 0,
\qquad
\E{\balpha\balpha^T} = I_{r_s},
\qquad
\E{(\bvecs-\bars)\balpha^T} = \Phis,
\\
\E{e_s} &= 0,
\qquad
\E{e_s\balpha^T} = 0,
\qquad
\E{e_s\widehat{\bvecs}^{T}} = 0.
\end{aligned}
\end{equation}
\end{lemma}

\begin{proof}
Equation~\eqref{eq:conditional_pythagoras_appendix} is the orthogonal-projection
identity for conditional expectation in \(L^2\). For the source identities,
\(\E{\bvecs-\bars}=0\) and \(\operatorname{Cov}(\bvecs)=\Rs^{-1}\) give
\[
\E{\balpha}=0,
\qquad
\E{\balpha\balpha^T}
=
\Phis^T\Rs\Rs^{-1}\Rs\Phis
=I_{r_s},
\qquad
\E{(\bvecs-\bars)\balpha^T}
=
\Rs^{-1}\Rs\Phis
=\Phis.
\]
Since \(e_s=(\bvecs-\bars)-\Phis\balpha\), the remaining identities follow:
\[
\E{e_s}=0,
\qquad
\E{e_s\balpha^T}=\Phis-\Phis=0,
\qquad
\E{e_s\widehat{\bvecs}^{T}}
=
\E{e_s}\bars^T+\E{e_s\balpha^T}\Phis^T
=0.
\]
\end{proof}

\section{Proof of the optimal blockwise ridge approximation}
\label{app:proof_blockwise_ridge_approximation}

\begin{proof}[Proof of \Cref{thm:optimal_blockwise_ridge_approximation}]
Up to almost-sure equality, \(\Bproj(\Pm,\Ps,\Py)\) is the closed subspace of
\(L^2_{\prior}\) consisting of \(\operatorname{Im}(\Py)\)-valued functions
measurable with respect to \(\sigma(\Pm\bvecm,\Ps\bvecs)\). Hence the
Hilbert-space projection theorem gives
\[
\FoptY
=
\E{\Py\Fwhite(\bvecm,\bvecs)\mid \Pm\bvecm,\Ps\bvecs}
=
\Py\,\E{\Gm\bvecs\mid \Pm\bvecm,\Ps\bvecs}.
\]
Because the prior blocks are independent and the conditioning variables depend
on separate blocks, their conditional law factorizes, so
\[
\E{\Gm\bvecs\mid \Pm\bvecm,\Ps\bvecs}
=
\E{\Gm\mid \Pm\bvecm}\,
\E{\bvecs\mid \Ps\bvecs}.
\]
With \(\bvecm_r=\Pm\bvecm\), \(\bvecs_r=\Ps\bvecs\), and the definitions of
\(\Gopt\) and \(\sopt\), this is
\[
\FoptY(\Ptheta\btheta)
=
\Py\,\Gopt(\bvecm_r)\,\sopt(\bvecs_r),
\]
which proves the result.
\end{proof}

\section{Proof of the optimal affine-in-source approximation}
\label{app:proof_optimal_affine_in_source}

\begin{proof}[Proof of \Cref{thm:optimal_affine_in_source}]
Let \(\bvecm_r=\Pm\bvecm\). By
\Cref{lem:projection_source_identities}, for square-integrable coefficient
functions \(a\) and \(B\),
\[
\Ex{\prior}{\|a(\bvecm_r)+B(\bvecm_r)\balpha\|^2}
=
\Ex{\priorm}{\|a(\bvecm_r)\|^2+\|B(\bvecm_r)\|_{\mathrm F}^2}.
\]
Thus \((a,B)\mapsto a+B\balpha\) is an isometry from the corresponding
Hilbert direct sum onto \(\Aaff(\Pm,\Ps,\Py)\); in particular, the affine
class is closed and its \(L^2_{\prior}\)-projection exists.

Every element of \(\Aaff\) lies in \(\operatorname{Im}(\Py)\). Orthogonality
of \(\Py\) therefore reduces the minimization of
\(\Ex{\prior}{\|\Fwhite-a-B\balpha\|^2}\) to that of
\[
\Ex{\prior}{
\|\Py\Gm\bvecs-a(\bvecm_r)-B(\bvecm_r)\balpha\|^2}.
\]
Since \(\balpha\) is independent of \(\bvecm_r\),
\Cref{lem:projection_source_identities} gives
\(\E{\balpha\mid\bvecm_r}=0\) and
\(\E{\balpha\balpha^T\mid\bvecm_r}=I_{r_s}\). The conditional affine
normal equations are therefore
\[
 a^*(\bvecm_r)
 =
 \E{\Py\Gm\bvecs\mid\bvecm_r},
\qquad
 B^*(\bvecm_r)
 =
 \E{\Py\Gm\bvecs\balpha^T\mid\bvecm_r}.
\]
Independence of the medium and source blocks and
\Cref{lem:projection_source_identities} give
\[
 a^*(\bvecm_r)
 =
 \Py\Gopt(\bvecm_r)\bars,
\qquad
 B^*(\bvecm_r)
 =
 \Py\Gopt(\bvecm_r)\E{\bvecs\balpha^T}
 =
 \Py\Gopt(\bvecm_r)\Phis.
\]
Hence
\[
\tilde{\Fwhite}_{\mathrm{aff}}^*(\Ptheta\btheta)
=
\Py\Gopt(\bvecm_r)
\bigl(\bars+\Phis\balpha\bigr)
=
\Py\Gopt(\bvecm_r)
\widehat{\bvecs}_{\mathrm{aff}}(\balpha).
\]
The covariance of \(\balpha\) is the identity, so the conditional affine
normal equations are nondegenerate and the minimizer is unique up to
almost-sure equality.

Finally, \(\Aaff(\Pm,\Ps,\Py)\subset\Bproj(\Pm,\Ps,\Py)\). For nested closed
subspaces, projecting first onto the larger subspace and then onto the smaller
one is the same as projecting directly onto the smaller subspace. Therefore the
same element is obtained by projecting the optimal blockwise ridge approximation
\(\FoptY\) onto \(\Aaff\).
\end{proof}

\section{Proof of exactness of the affine restriction for Gaussian source priors}
\label{app:proof_gaussian_affine_restriction_exactness}

\begin{proof}[Proof of \Cref{cor:gaussian_affine_restriction_exactness}]
Write
\(
\bvecs-\bars
=
\Ps(\bvecs-\bars)+(I-\Ps)(\bvecs-\bars)
\).
Because \(\bvecs\sim\N{\bars}{\Rs^{-1}}\) and \(\Ps\) is
\(\Rs\)-orthogonal, the two Gaussian components are uncorrelated:
\[
\operatorname{Cov}\!\left(
\Ps(\bvecs-\bars),(I-\Ps)(\bvecs-\bars)
\right)
=
\Ps\Rs^{-1}(I-\Ps)^T
=
\Phis\Phis^T(I-\Ps)^T
=0.
\]
They are therefore independent. The discarded component has zero mean, so
\[
\sopt(\Ps\bvecs)
=
\E{\bvecs\mid\Ps\bvecs}
=
\bars+\Ps(\bvecs-\bars)
=
\bars+\Phis\balpha.
\]
This proves \eqref{eq:gaussian_source_conditional_mean}; substitution into
\Cref{thm:optimal_affine_in_source} gives
\eqref{eq:gaussian_affine_equals_blockwise_ridge}. Therefore the
optimal affine and blockwise ridge approximations coincide.
\end{proof}

\section{Proof of the direct affine error decomposition}
\label{app:proof_affine_error_decomposition}

\begin{proof}[Proof of \Cref{prop:affine_error_decomposition}]
Set
\[
\widehat{\bvecs}
\coloneqq
\widehat{\bvecs}_{\mathrm{aff}}(\balpha)
=
\bars+\Phis\balpha,
\qquad
 e_s\coloneqq\bvecs-\widehat{\bvecs}.
\]
Using the optimal affine-in-source approximation, decompose
\begin{align*}
\Fwhite-\tilde{\Fwhite}_{\mathrm{aff}}
&=
\underbrace{(I-\Py)\Fwhite}_{r_{\mathrm{obs}}}
+
\underbrace{\Py\Gm e_s}_{r_{\mathrm{src}}}
+
\underbrace{\Py\bigl(\Gm-\Gopt(\bvecm_r)\bigr)
\widehat{\bvecs}}_{r_{\mathrm{med}}}
+
\underbrace{\tilde{\Fwhite}_{\mathrm{aff}}^*
-\tilde{\Fwhite}_{\mathrm{aff}}}_{r_{\mathrm{surr}}}.
\end{align*}
These four terms are pairwise orthogonal in \(L^2_{\prior}\). The output term
is pointwise orthogonal to the other three because they are
\(\operatorname{Im}(\Py)\)-valued. By
\eqref{eq:source_moment_identities_appendix} and independence of the medium
and source blocks, \(r_{\mathrm{src}}\) is orthogonal to both
\(r_{\mathrm{med}}\) and \(r_{\mathrm{surr}}\), the latter having the form
\(a_\Delta(\bvecm_r)+B_\Delta(\bvecm_r)\balpha\). Finally,
\(r_{\mathrm{surr}}\) is measurable with respect to
\((\bvecm_r,\balpha)\), while independence gives
\[
\E{\Gm-\Gopt(\bvecm_r)\mid\bvecm_r,\balpha}
=
\E{\Gm-\Gopt(\bvecm_r)\mid\bvecm_r}
=0,
\]
so \(r_{\mathrm{med}}\perp r_{\mathrm{surr}}\). Pythagoras therefore yields
\[
\Ex{\prior}{
\|\Fwhite-\tilde{\Fwhite}_{\mathrm{aff}}\|^2}
=
\Robs+
\Ex{\prior}{\|\Py\Gm e_s\|^2}
+
\Cmedaff+
\Rsurraff.
\]
Since \(e_s=(I-\Ps)(\bvecs-\bars)\) and
\(\operatorname{Cov}(\bvecs)=\Rs^{-1}\), cyclicity of the trace gives
\begin{align*}
\Ex{\prior}{\|\Py\Gm e_s\|^2}
&=
\tr\!\left(
\Hss(\Py)(I-\Ps)\Rs^{-1}(I-\Ps^T)
\right)\\
&=
\tr\!\left(
\Rs^{-1}(I-\Ps^T)\Hss(\Py)(I-\Ps)
\right)
=
\Rsrcaff.
\end{align*}
This proves \eqref{eq:affine_direct_mse_split} and
\eqref{eq:Rsrc_aff_exact}.
\end{proof}

\section{Proof of the affine medium trace bound and expected-KL control}
\label{app:proof_affine_medium_control_kl}

\begin{proof}[Proof of \Cref{prop:affine_trace_kl_bounds}]
Linearity of \(v\mapsto J_v(\bvecm)\) gives
\[
J_{\widehat{\bvecs}_{\mathrm{aff}}(\balpha)}(\bvecm)
=
J_{\bars}(\bvecm)
+
\sum_{i=1}^{r_s}\alpha_i J_{\bm\phi_i}(\bvecm).
\]
Using \(\E{\balpha}=0\) and
\(\E{\balpha\balpha^T}=I_{r_s}\) from
\Cref{lem:projection_source_identities}, expansion of the quadratic form
immediately gives \eqref{eq:Hmmaff_moment_expansion}.

For fixed \(\balpha\), let
\(h_{\balpha}(\bvecm)=\Py\Gm
\widehat{\bvecs}_{\mathrm{aff}}(\balpha)\). By
\Cref{ass:standing_assumptions,ass:medium_poincare} and Fubini's theorem, the subspace
Poincar\'e inequality applies for almost every \(\balpha\). Independence yields
\[
\E{h_{\balpha}(\bvecm)\mid\Pm\bvecm}
=
\Py\Gopt(\Pm\bvecm)
\widehat{\bvecs}_{\mathrm{aff}}(\balpha),
\]
and hence
\begin{align*}
&\Ex{\bvecm\sim\priorm}{
\left\|
\Py\bigl(\Gm-\Gopt(\Pm\bvecm)\bigr)
\widehat{\bvecs}_{\mathrm{aff}}(\balpha)
\right\|^2}\\
&\qquad\le
\tr\!\left(
\Rm^{-1}(I-\Pm^T)
\Ex{\bvecm\sim\priorm}{
J_{\widehat{\bvecs}_{\mathrm{aff}}(\balpha)}(\bvecm)^T
\Py
J_{\widehat{\bvecs}_{\mathrm{aff}}(\balpha)}(\bvecm)}
(I-\Pm)
\right).
\end{align*}
Averaging over the source prior gives
\[
\Cmedaff
\le
\tr\!\left(
\Rm^{-1}(I-\Pm^T)
\Hmmaff(\Py,\Ps)
(I-\Pm)
\right)
=
\Tmedaff.
\]
The eigenspace claim follows from \eqref{eq:trace_trailing_ev_equality}.
The medium bound, together with \eqref{eq:prior_predictive_kl_forward_bound}
and \eqref{eq:affine_direct_mse_split}, yields
\[
\Ex{\yobs}{
D_{\mathrm{KL}}\!\left(
\post\,\|\,\tilde{\pi}_{\mathrm{post}}^{\yobs}
\right)}
\le
\frac12\left(
\Robs+\Rsrcaff+\Tmedaff+\Rsurraff
\right),
\]
which is \eqref{eq:affine_expected_kl_bound}.
\end{proof}

\section{Proof of the expected coefficient-loss identity}
\label{app:proof_expected_coeff_loss}

\begin{proof}[Proof of \Cref{prop:coefficient_loss_identity}]
Define
\[
q_0(\bvecm)\coloneqq\Phiy^T\Gm\bars,
\qquad
Q_0(\bvecm)\coloneqq\Phiy^T\Gm\Phis.
\]
By the definition of \(\Gopt_r\),
\[
\qsurr^*(\bmu)=\E{q_0(\bvecm)\mid\bmu},
\qquad
\Qsurr^*(\bmu)=\E{Q_0(\bvecm)\mid\bmu},
\]
where the matrix identity is understood componentwise. Applying
\eqref{eq:conditional_pythagoras_appendix} to \(q_0\) and to
\(\operatorname{vec}Q_0\) gives
\begin{align*}
\mathcal L_{\mathrm{coeff}}(\bm w)
&=
\Ex{\bvecm\sim\priorm}{
\|\qsurr(\bmu;\bm w)-\qsurr^*(\bmu)\|^2
+
\|\Qsurr(\bmu;\bm w)-\Qsurr^*(\bmu)\|_{\mathrm F}^2}
+
\Cmedaff,
\end{align*}
Since \(\bmu\) and \(\bvecm_r\) carry the same retained medium coordinates,
the residual term is exactly \eqref{eq:Cmed_aff_def}.

It remains to identify the first expectation. Set
\(\Delta q=\qsurr^*-\qsurr\) and
\(\Delta Q=\Qsurr^*-\Qsurr\). Using
\(\Phiy^T\Phiy=I\), the definitions of the optimal and trainable affine
surrogates, independence of the prior blocks, and
\eqref{eq:source_moment_identities_appendix},
\[
\Rsurraff(\bm w)
= \Ex{\prior}{\|\Delta q(\bmu)+\Delta Q(\bmu)\balpha\|^2}
= \Ex{\bvecm\sim\priorm}{
\|\Delta q(\bmu)\|^2+
\|\Delta Q(\bmu)\|_{\mathrm F}^2}.
\]
Consequently,
\(
\mathcal L_{\mathrm{coeff}}(\bm w)
=
\Rsurraff(\bm w)+\Cmedaff
\), proving \eqref{eq:coefficient_loss_decomposition}.

For completeness, let
\(A(\bvecm)=\Gm-\Gopt(\bvecm_r)\). Since
\(\widehat{\bvecs}_{\mathrm{aff}}=\bars+\Phis\balpha\), the same source
moment identities and \(\Py=\Phiy\Phiy^T\) give
\begin{align*}
\Ex{\prior}{
\|\Py A(\bvecm)\widehat{\bvecs}_{\mathrm{aff}}(\balpha)\|^2}
&=
\Ex{\bvecm\sim\priorm}{
\|\Phiy^TA(\bvecm)\bars\|^2
+
\|\Phiy^TA(\bvecm)\Phis\|_{\mathrm F}^2},
\end{align*}
which proves the equivalence of \eqref{eq:Cmed_aff_def} and
\eqref{eq:Cmed_aff_equiv}.
\end{proof}

\section{Quadratic benchmark construction and analytic formulas}
\label{app:quadratic_benchmark_formulas}

This appendix gives the coefficient-generation procedure, parameter regimes,
Monte Carlo evaluation details, and exact formulas used in
Section~\ref{sec:toy_validation}. The benchmark map, basis constructions, and
rank-matching rule are defined in the main text.

\subsection{Coefficient generation and parameter regimes}
\label{app:quadratic_benchmark_details}

The quadratic benchmark map and dimensions are defined in
\eqref{eq:quadratic_benchmark_map}. Its coefficient matrices \(\mathbf{A}_0,\mathbf{A}_i,\mathbf{B}_{ij}\in \mathbb{R}^{n_{\mathrm{obs}}\times d_s}\) are generated from common random output and source rotations \(U\) and \(V\). Specifically, each coefficient matrix \(\mathbf C\) is generated as
\begin{equation}
\label{eq:toy_coefficient_generation}
\mathbf C =
U\,\operatorname{diag}\!\left(
0.35\,\gamma\,0.9^{\ell-1}\exp(0.25z_\ell)
\right)_{\ell=1}^{10}V^T,
\qquad
z_\ell\stackrel{\mathrm{iid}}{\sim}\mathcal N(0,1),
\end{equation}
using \(\gamma=\gamma_0\) for \(\mathbf A_0\) and \(\gamma=1\) for all \(\mathbf A_i\) and \(\mathbf B_{ij}\). For the quadratic coefficients, we independently generate \(\mathbf B_{ij}\) for \(1\le i\le j\le d_m\) using \eqref{eq:toy_coefficient_generation} and then set \(\mathbf B_{ji}=\mathbf B_{ij}\) for \(i<j\). This construction enforces the symmetry \(\mathbf B_{ij}=\mathbf B_{ji}\) used in \eqref{eq:quadratic_benchmark_map}. The factor \(0.9^{\ell-1}\) sets the gradual spectral decay.

The medium prior is \(\bvecm\sim\N{\bm 0}{I_{d_m}}\), so the canonical medium Poincar\'e metric matrix is \(\Rm=I_{d_m}\). For Gaussian sources, let \(\bars=c_s v\), with \(v_i=0.1(i+1)\), \(i=1,\ldots,10\), and use \( (c_s,\gamma_0)\in\{(2,2),(4,0.5),(8,0.3)\}  \)
for the source-dominated, approximately balanced, and medium-dominated regimes. For the Gaussian-mixture experiment, use
\[
\priors(\bvecs)
=
\tfrac12\mathcal N(\mu_+,I_{10})
+
\tfrac12\mathcal N(\mu_-,I_{10}),
\qquad
\mu_\pm=c_b v\pm\delta\frac{w}{\|w\|},
\]
where \(w_i=1-(i-1)/6\), and use \((c_b,\delta,\gamma_0)=(2,2,2)\), \((4,3,0.5)\), and \((8,4,0.3)\) for the same three regimes.

For the numerical evaluation of the joint projected errors, we use
\(4\times10^4\) joint samples for Gaussian sources and \(2\times10^4\) for
Gaussian-mixture sources. The blockwise affine errors and medium bound are evaluated
analytically.

For every orthogonal medium projector \(\Pm\),
\begin{equation*}
\Gopt(\Pm\bvecm)
=
\mathbf{A}_0
+\sum_{i=1}^{d_m}(\Pm\bvecm)_i\mathbf{A}_i
+\sum_{i,j=1}^{d_m}
\left[
\bigl((\Pm\bvecm)(\Pm\bvecm)^T\bigr)_{ij}
+(I-\Pm)_{ij}
\right]\mathbf{B}_{ij}.
\end{equation*}

For the case of Gaussian source prior , if \(\bvecs\sim\N{\bars}{I_{d_s}}\), then
\(\widehat{\bvecs}_{\mathrm{aff}}=\bars+\Ps(\bvecs-\bars)\)
is also the retained conditional source mean.

For the case of Gaussian-mixture source prior, we have
\begin{equation}
\label{eq:toy_gmm_prior}
\priors(\bvecs)
=
\sum_{k=1}^{K}w_k\,\mathcal N(\mu_k,I_{d_s}),
\qquad
w_k>0,
\qquad
\sum_{k=1}^{K}w_k=1,
\end{equation}
define
\begin{equation*}
\bars=\sum_{k=1}^{K}w_k\mu_k,
\qquad
\bm C_s
=I_{d_s}+\sum_{k=1}^{K}w_k\mu_k\mu_k^T-\bars\bars^T,
\qquad
\Rs=\bm C_s^{-1}.
\end{equation*}
Here \(\Rs\) is the inverse covariance of the full mixture prior, not the
precision of an individual mixture component. We use the main-text convention
\(
\balpha=\Phis^T\Rs(\bvecs-\bars)
\),
\(
\Ps=\Phis\Phis^T\Rs
\), with \(\Phis^T\Rs\Phis=I\), and hence
\[
\widehat{\bvecs}_{\mathrm{aff}}(\balpha)
=
\bars+\Phis\balpha
=
\bars+\Ps(\bvecs-\bars).
\]
The SPPL affine construction uses these first two source moments and the retained
source space; it does not require the nonlinear retained conditional mean.

\subsection{Exact unrestricted joint benchmark for the Gaussian-mixture source prior}
\label{app:toy_gmm_joint_benchmark}

The formulas also cover the Gaussian benchmark with \(K=1\), \(w_1=1\), and \(\mu_1=\bars\). Let
\[
\bm{\Phi}_\theta^T\bm R_{\btheta}\bm{\Phi}_\theta=I_{r_{\mathrm{joint}}},
\qquad
P_{\mathrm{joint}}
=\bm{\Phi}_\theta\bm{\Phi}_\theta^T\bm R_{\btheta},
\qquad
\xi
=\bm{\Phi}_\theta^T\bm R_{\btheta}(\btheta-\bar{\btheta}),
\]
where \(\bar{\btheta}=(\barm,\bars)\) and
\(\bm R_{\btheta}=\diag(\Rm,\Rs)\), consistently with the generalized-eigenvector construction above. Define
\[
A\coloneqq\bm{\Phi}_\theta^T\bm R_{\btheta},
\qquad
S_\xi\coloneqq AA^T
=\bm{\Phi}_\theta^T\bm R_{\btheta}^2\bm{\Phi}_\theta.
\]
Introduce \(c\in\{1,\ldots,K\}\) with \(\mathbb P(c=k)=w_k\) and \(\bvecs\mid c=k\sim\N{\mu_k}{I_{d_s}}\). Then
\[
\btheta\mid c=k\sim\N{\bar{\btheta}_k}{I_{d_m+d_s}},
\qquad
\bar{\btheta}_k=\begin{bmatrix}\bm 0\\\mu_k\end{bmatrix},
\qquad
\bm a_k\coloneqq A(\bar{\btheta}_k-\bar{\btheta}),
\]
and standard Gaussian conditioning gives
\begin{equation*}
\btheta\mid(\xi,c=k)
\sim
\N{
\bar{\btheta}_k
+A^T S_\xi^{-1}(\xi-\bm a_k)
}{
I_{d_m+d_s}-A^T S_\xi^{-1}A
}.
\end{equation*}
Write the conditional mean and covariance as
\[
\bar{\btheta}_k+A^T S_\xi^{-1}(\xi-\bm a_k)
=
\begin{bmatrix}\bmu_{m,k}(\xi)\\\bmu_{s,k}(\xi)\end{bmatrix},
\qquad
I_{d_m+d_s}-A^T S_\xi^{-1}A
=
\begin{bmatrix}
\Sigma_{mm}^{\perp}&\Sigma_{ms}^{\perp}\\
\Sigma_{sm}^{\perp}&\Sigma_{ss}^{\perp}
\end{bmatrix}.
\]
The conditional mixture weights are
\[
\pi_k(\xi)
=
\frac{
w_k\exp\!\left[
-\tfrac12(\xi-\bm a_k)^T S_\xi^{-1}(\xi-\bm a_k)
\right]
}{
\sum_{j=1}^{K}w_j\exp\!\left[
-\tfrac12(\xi-\bm a_j)^T S_\xi^{-1}(\xi-\bm a_j)
\right]
}.
\]
With \(\mu_{m,k,i}\) denoting component \(i\) of \(\bmu_{m,k}\), Gaussian moment identities give the componentwise conditional forward mean directly as
\begin{align*}
f_k(\xi)
&\coloneqq \E{\G(\bvecm)\bvecs\mid\xi,c=k}\\
&=
\mathbf A_0\bmu_{s,k}
+
\sum_{i=1}^{d_m}\mathbf A_i
\left[
\mu_{m,k,i}\bmu_{s,k}
+(\Sigma_{ms}^{\perp})_{i,:}^T
\right]
\\
&\quad+
\sum_{i,j=1}^{d_m}\mathbf B_{ij}
\Bigl[
\bigl(\mu_{m,k,i}\mu_{m,k,j}+(\Sigma_{mm}^{\perp})_{ij}\bigr)\bmu_{s,k}
+\mu_{m,k,i}(\Sigma_{ms}^{\perp})_{j,:}^T
+\mu_{m,k,j}(\Sigma_{ms}^{\perp})_{i,:}^T
\Bigr],
\end{align*}
where the \(\xi\)-dependence of \(\bmu_{m,k}\) and \(\bmu_{s,k}\) is suppressed in the last display. Hence the exact unrestricted joint benchmark is
\(\Fopt_{\mathrm{joint}}(\xi) = \E{\Fwhite(\btheta)\mid\xi} = \sum_{k=1}^{K}\pi_k(\xi)f_k(\xi)\).

\section{Coseismic model, prior, truth, and offline-construction details}
\label{app:coseismic_details}

\subsection{Physical model and discretization}

We adopt the two-dimensional weakly symmetric Hellinger--Reissner elasticity formulation and fault treatment of \cite{Puel22}, as used in the heterogeneous-material/fault-slip joint-inversion framework of \cite{puel2024adjoint_joint}. Let \(\Omega\subset\mathbb R^2\) have external Dirichlet and Neumann boundaries \(\Gamma_D\) and \(\Gamma_N\) and an internal fault interface \(\Gamma_F\). With zero body force, the state consists of displacement \(\bm u\), stress \(\bm\sigma\), and rotation \(\bm\varrho=\frac12(\nabla\bm u-(\nabla\bm u)^T)\), satisfying
\begin{equation*}
\begin{alignedat}{3}
\bm A(m)\bm\sigma
&= \nabla\bm u-\bm\varrho,
&\qquad
-\nabla\cdot\bm\sigma
&= \bm 0,
&\qquad& \text{in }\Omega,
\\[2pt]
\bm u
&= \bm 0 \quad \text{on }\Gamma_D,
&\qquad
\bm\sigma\bm n
&= \bm 0 \quad \text{on }\Gamma_N,
\\[2pt]
\left[\!\left[T(\bm n_F)\bm u\right]\!\right]
&= s,
&\qquad
\left[\!\left[\bm u\cdot\bm n_F\right]\!\right]
&= 0,
&\qquad
\left[\!\left[\bm\sigma\bm n_F\right]\!\right]
&= \bm 0
\quad \text{on }\Gamma_F .
\end{alignedat}
\end{equation*}
Here \(\left[\!\left[v\right]\!\right]=v^- - v^+\) denotes the jump across the fault, \(\bm n_F=\bm n^+=-\bm n^-\) is a fixed fault normal, and \(T(\bm n_F)\) extracts the tangential displacement component. Thus \(s\) prescribes tangential slip, while normal displacement and traction remain continuous across the fault.

We fix the Poisson ratio at \(\nu=0.25\) and parameterize the shear modulus by
\[
\bm A(m)\bm\sigma
=
\frac{1}{2\mu(m)}
\left[\bm\sigma-\nu\tr(\bm\sigma)\bm I\right],
\qquad
\mu(m)=\qty{30}{\giga\pascal}\,[1.1+\tanh(m)],
\]
which is the \(d=2\) specialization of the isotropic compliance law used in \cite{Puel_2024}.

The domain is a \(500\times250\) km vertical cross-section with a dipping fault. Horizontal and vertical surface displacements are observed at 20 GPS stations, giving \(n_{\mathrm{obs}}=40\). Following \cite{Puel22,Puel_2024}, the mixed state is discretized with the second-order stable \(\mathrm{BDM}_2\times\mathrm{DG}_1\times\mathrm{DG}_1\) finite-element triplet, while the medium and fault-slip fields use continuous piecewise-linear Lagrange elements. The resulting state dimension is \(145083\), and the medium and source parameter dimensions are \(d_m=3101\) and \(d_s=29\), respectively. We use \(r_m=r_s=12\) and \(r_y=17\). The observation-noise standard deviation is \(0.015\).

After discretization, let \(\bvecm\) and \(\bvecs\) denote the medium and slip coefficient vectors. The mixed state equation has the form \(\mathbf K(\bvecm)\bm x=\mathbf S\bvecs\), where \(\bm x=(\bm\sigma,\bm u,\bm\varrho)\) is the discrete mixed state and \(\mathbf S\) is the fixed linear fault-slip injection operator. Hence
\[
\Fwhite(\bvecm,\bvecs)
=
\Signoise^{-1/2}\mathcal B\mathbf K(\bvecm)^{-1}\mathbf S\bvecs
=
\G(\bvecm)\bvecs,
\]
so the forward map is nonlinear in the medium and linear in the fault slip. 

\subsection{Priors, reduced spaces, and synthetic data}

The medium prior uses an anisotropic BiLaplacian precision with \((\gamma_m,\delta_m)=(10,8\times10^{-2})\), Robin boundary terms, and nominal correlation lengths \(11.2\) in the transverse/far-field direction and \(89.4\) along the fault. Because this prior is Gaussian, its precision is also the canonical medium Poincar\'e metric matrix \(\Rm\). The slip prior uses a BiLaplacian precision with \((\gamma_s,\delta_s)=(1.8,2\times10^{-3})\), Robin boundary terms, and nominal correlation length \(30\); this precision is \(\Rs\).

The common block reduced spaces and the rank-selection rule used in the three-model ablation are defined in \Cref{sec:coseismic_protocol}. Here we record their shared PDE construction cost. The expectations defining \(\Cy\) and \(\Hss(\Py)\) use the same \(200\) independent medium-prior samples. For each sample, all \(d_s=29\) source directions are evaluated once and reused to assemble both matrices, for \(200d_s\) forward right-hand-side solves. After selecting the rank-\(r_s=12\) source space, \(\Hmmaff(\Py,\Ps)\) is estimated with \(200\) independent medium-prior samples. For each such sample, the zero-mean specialization requires \(r_s=12\) forward solves and \(r_y=17\) adjoint solves, with all right-hand sides reusing the same PDE-operator factorization. Because the source prior is Gaussian, this affine medium diagnostic matrix is the medium criterion used when the affine restriction is exact. This reduced-space construction is performed once and shared by Matrix-loss SPPL, Probe-loss SPPL, and DIPNet; hence its cost is common to all pairwise surrogate comparisons, while remaining part of the reported end-to-end offline cost.

Noise-free synthetic observations are generated by a forward solve evaluated at the prescribed medium and slip truths, rather than at the MAP. The likelihood nevertheless uses
the prescribed nonzero Gaussian observation-noise covariance. We denote the resulting
nominal whitened datum by
\(\yobs_1=\Fwhite(m^\dagger,s^\dagger)\). The truth fields are
\begin{align}
m^\dagger(x,y) &= 0.25\, \sin\!\left( 2\pi(1.2)\frac{x}{500} \right) \cos\!\left( 2\pi(1.5)\frac{y}{300} \right),
\label{eq:coseismic_medium_truth} \\
s^\dagger(y) &= -10\exp\!\left( -\frac{(y+20)^2}{15^2} \right),
\label{eq:coseismic_slip_truth}
\end{align}
where \(x\) and \(y\) are measured in kilometers.
The likelihood uses additive Gaussian noise with covariance \(0.015^2 I_{40}\). Any reported percentage noise level is defined relative to the RMS as \(100\times0.015/\operatorname{rms}(\dobs^\dagger)\%\), rather than relative to the maximum displacement.

\section{Neural-surrogate training and data splits}
\label{app:network_training_details}

\subsection{Matched-cost training data}

All three surrogate variants use training budgets
\(N_m\in\{100,400,1600,6400\}\), with the same medium realizations for a
given \(N_m\). The block reduced spaces are constructed once as described in
\Cref{app:coseismic_details} and are shared by all three methods. For
Matrix-loss SPPL, each medium sample provides one coefficient-matrix target
\(\Qproj_j\), assembled using \(r_s=12\) forward right-hand-side solves. For
Probe-loss SPPL and DIPNet, each medium sample \(\bvecm_j\) is paired with
\(r_s=12\) independent source probes
\(\bvecs_{jk}\stackrel{\mathrm{iid}}{\sim}\priors\) and their projected
outputs. Exactly the same realized probe--output pairs are used by the two
probe-trained methods. Thus all three methods use the same number of PDE
right-hand-side solves per medium realization.

All methods use the same 200 validation and 400 test medium realizations. For
each validation or test medium realization, the probe-trained methods use 12
independently sampled source probes, shared between Probe-loss SPPL and DIPNet,
giving 2400 validation and 4800 test pairs. These unscaled probe data are used for the nominal Probe-loss SPPL and for the
\(c_{\mathrm{train}}=1\) DIPNet in the matched-cost ablation.

\subsection{Probe-loss SPPL and its relation to Matrix-loss SPPL}
\label{app:probe_loss_population}

For each medium sample \(\bvecm_j\), let
\(\bvecs_{jk}\stackrel{\mathrm{iid}}{\sim}\priors\),
\(k=1,\ldots,r_s\), denote the source probes and define
\(
\balpha_{jk}=\Phis^*(\bvecs_{jk}-\bars),
\bm z_{jk}=\Phiy^T\Fwhite(\bvecm_j,\bvecs_{jk}).
\)
Probe-loss SPPL minimizes
\begin{equation}
\label{eq:empirical_probe_loss}
\widehat{\mathcal L}_{\mathrm{probe}}(\bm w)
=
\frac{1}{N_m r_s}
\sum_{j=1}^{N_m}\sum_{k=1}^{r_s}
\left\|
\qsurr(\bmu_j;\bm w)+\Qsurr(\bmu_j;\bm w)\balpha_{jk}
-
\bm z_{jk}
\right\|^2.
\end{equation}
In the zero-mean coseismic setting, \(\bars=0\) and \(\qsurr\equiv0\).

Under the standing source-moment assumptions, for fixed \(\bvecm\), write
\[
e_s=(I-\Ps)(\bvecs-\bars).
\]
Then
\[
\Phiy^T\Fwhite(\bvecm,\bvecs)
=
\qproj(\bvecm)+\Qproj(\bvecm)\balpha+\Phiy^T\Gm e_s.
\]
The source-moment identities in \Cref{lem:projection_source_identities} give
\(\E{e_s}=0\) and \(\E{e_s\balpha^T}=0\). Consequently,
\begin{equation}
\label{eq:probe_matrix_population_identity}
\begin{aligned}
&\Ex{\bvecs\sim\priors}{
\left\|
\qsurr(\bmu;\bm w)+\Qsurr(\bmu;\bm w)\balpha
-\Phiy^T\Fwhite(\bvecm,\bvecs)
\right\|^2}
\\
&\qquad =
\left\|\qsurr(\bmu;\bm w)-\qproj(\bvecm)\right\|^2
+
\left\|\Qsurr(\bmu;\bm w)-\Qproj(\bvecm)\right\|_{\mathrm F}^2
+
c_{\mathrm{src}}(\bvecm),
\end{aligned}
\end{equation}
where
\[
c_{\mathrm{src}}(\bvecm)
\coloneqq
\Ex{\bvecs\sim\priors}{
\left\|
\Phiy^T\Gm(I-\Ps)(\bvecs-\bars)
\right\|^2}
\]
is independent of \(\bm w\). Hence the population Probe-loss and Matrix-loss
objectives have the same minimizer, and differences between their empirical
results reflect finite source-probe sampling.

\subsection{Network architecture, optimization, and replication}

Because \(\bars=0\), Matrix-loss SPPL and Probe-loss SPPL use the same network
\(\Qsurr:\mathbb R^{12}\to\mathbb R^{17\times12}\), with
\(\qsurr\equiv0\). DIPNet uses the generic projected network
\(\mathcal N_{\mathrm{DIPNet}}:\mathbb R^{24}\to\mathbb R^{17}\). All
networks have five hidden layers with 256 neurons per layer and GELU
activations. This gives 318,924 trainable parameters for each SPPL network and
273,937 for DIPNet; depth and hidden width, rather than total parameter count,
are matched.

DIPNet inputs are standardized featurewise, whereas SPPL standardizes only the medium coordinates \(\bmu\), leaving \(\balpha\) unchanged and restoring the coefficient matrix (including its empirical mean) before multiplication by \(\balpha\). This ensures that \(\tilde F_{\mathrm{SPPL}}(\bmu,0)=0\) and \(\tilde F_{\mathrm{SPPL}}(\bmu,c\balpha)=c\tilde F_{\mathrm{SPPL}}(\bmu,\balpha)\) for any network weights. Outputs are centered and scaled by a single global training-set
standard deviation, applied to matrix entries for Matrix-loss SPPL and to
projected output vectors for the probe-trained methods. Training uses Adam with
initial learning rate \(10^{-3}\), ReduceLROnPlateau with factor \(0.5\) and
patience 10, and early-stopping patience 20.

Mini-batches are organized by 50 medium indices. For each selected medium
sample, Matrix-loss SPPL uses the full coefficient-matrix target, whereas each
probe-trained method uses all 12 associated probe pairs. Within each training
seed, all methods use the same ordered sequence of medium-index batches. Results
are reported over three paired seeds, with no best-seed selection; the medium
samples, source probes, and targets are fixed across seeds, while network
initialization and batch ordering vary. Figure~\ref{fig:ood_rmse} reports the
across-seed mean and minimum--maximum envelope; other diagnostics use the
reporting convention stated in their captions.

\subsection{Source-scale DIPNet variants}

For the forward-level source-scale diagnostic in Figure~\ref{fig:ood_rmse}, we train separate DIPNets at three scales: \(c_{\mathrm{train}}=1/10\), \(1\), and \(10\). We repeat this for \(N_m=100\), \(400\), \(1600\), and \(6400\). For a fixed \(N_m\), the same medium
realizations are used across the three training scales, with the associated
source probes and projected outputs multiplied by the corresponding scale
factor. The \(c_{\mathrm{train}}=1\) model is the nominal DIPNet used in the
three-model ablation; the other two models are auxiliary scale-transfer
variants.

For the posterior source-scale diagnostic in Figure~\ref{fig:forward_error_source_scale}(b),
we use the \(N_m=6400\) models from this family and add a Combined DIPNet
trained on the union of the three scale-specific training sets. The Combined
model therefore uses three times the probe data and is not part of the
matched-cost three-model ablation.

\section{Coseismic posterior-evaluation and robustness protocol}
\label{app:coseismic_evaluation_protocol}

This section records the posterior-sampling and evaluation details used for the
coseismic surrogate comparisons in \Cref{sec:coseismic_protocol}.

\subsection{Retained-coordinate sampling, posterior accuracy, and online efficiency}

For the Gaussian-prior coseismic experiments, MCMC samples the reduced-space approximate posterior in \eqref{eq:sppl_reduced_posterior} using a MAP-based
likelihood-informed pCN proposal, denoted LA-pCN, adapted from Cui, Law, and
Marzouk~\cite{cui2016dimension} with the underlying function-space pCN proposal
of Cotter et al.~\cite{cotter2013mcmc}.

The bounds in \Cref{sec:sppl_reduction_theory} control the prior-predictive expectation of
the posterior KL divergence. The numerical values reported for the observed
inverse problems are instead fixed-\(\yobs\) posterior diagnostics based on the fixed-data forward KL divergence \(D_{\mathrm{KL}}(\post\,\|\,\tilde{\pi}_{\mathrm{post}}^{\yobs})\):
they assess posterior fidelity empirically for each realized data set rather than
serving as pointwise certificates from the data-free theory. Using posterior samples
\(\{\btheta_i\}_{i=1}^N\) from the high-fidelity reference chain, we estimate
\begin{equation}
\label{eq:fixed_data_posterior_kl_estimator}
\begin{aligned}
D_{\mathrm{KL}}\left(\post \| \tilde{\pi}_{\mathrm{post}}^{\yobs}\right)
\approx {}&\frac{1}{N} \sum_{i=1}^N
\left(\tilde{\mathcal{J}}^{\yobs}(\btheta_i) - \mathcal{J}^{\yobs}(\btheta_i)\right)\\
&+\log \left(\frac{1}{N} \sum_{i=1}^N \exp \left( \mathcal{J}^{\yobs}(\btheta_i) - \tilde{\mathcal{J}}^{\yobs}(\btheta_i)\right) \right).
\end{aligned}
\end{equation}
Here \(\mathcal J^{\yobs}\) and \(\tilde{\mathcal J}^{\yobs}\) are the
high-fidelity and surrogate negative log-posteriors, respectively. Since the
two posteriors use the same prior, their difference is determined by the
corresponding high-fidelity and surrogate likelihood potentials. In practice,
the estimator therefore reuses the high-fidelity posterior samples and
evaluates both forward models at those same samples.
For surrogate runs, the multivariate potential scale reduction factor
(MPSRF)~\cite{brooks1998general} over retained coordinates was at most
\(\num{1.001}\).

In addition to posterior fidelity and convergence, we assess online efficiency
using forward-evaluation cost, wall-clock cost per MCMC step, and integrated
autocorrelation time (IACT). These diagnostics quantify the acceleration
obtained by sampling the reduced surrogate posterior rather than the
high-fidelity PDE posterior.

\subsection{Source-scale robustness}

For the five data scales \(c\) defined in \Cref{sec:coseismic_protocol}, we run a separate high-fidelity PDE-based MCMC chain targeting \(\pi_{\mathrm{post}}^{\yobs_c}\). Let
\(\{\btheta_i^{(c)}\}_{i=1}^N\) denote samples from the corresponding reference chain. For each trained surrogate, we evaluate the high-fidelity and surrogate negative log-posteriors at these same samples and apply
\eqref{eq:fixed_data_posterior_kl_estimator}, with \(\yobs\) replaced by
\(\yobs_c\), to estimate
\[
D_{\mathrm{KL}}\!\left(
\pi_{\mathrm{post}}^{\yobs_c}
\,\|\,
\tilde{\pi}_{\mathrm{post}}^{\yobs_c}
\right).
\]
Thus every point in Figure~\ref{fig:forward_error_source_scale}(b) uses samples from the high-fidelity posterior corresponding to that test datum; the surrogate is used only to evaluate the approximate likelihood/posterior potential at those reference samples.

Figure~\ref{fig:forward_error_source_scale}(b) applies this posterior procedure to the
\(N_m=6400\) DIPNet source-scale family defined in
\Cref{app:network_training_details}. Matrix-loss SPPL, the three separately
trained DIPNets, and the Combined DIPNet are evaluated at every
test scale \(c\). Figure~\ref{fig:ood_rmse} is broader at the forward-model
level: it evaluates the three scale-specific DIPNets for every
\(N_m\in\{100,400,1600,6400\}\), together with the two SPPL references.

\section*{Acknowledgments}

The authors acknowledge the Texas Advanced Computing Center (TACC) at The University of Texas at Austin for providing computational resources that contributed to the research reported in this paper. Computations were performed on the Stampede3 system, which is supported by the U.S. National Science Foundation under Award No.~2320757.

\bibliographystyle{siamplain}
\bibliography{references}
\end{document}